\documentclass{amsart}%
\usepackage{amssymb}
\usepackage{amsmath}
\usepackage{amsfonts}
\usepackage[all,cmtip]{xy}
\usepackage[margin=1.6in]{geometry}
\usepackage[backref]{hyperref}%
\usepackage{graphicx}

\newtheorem{theorem}{Theorem}[section]
\theoremstyle{plain}
\newtheorem*{acknowledgement}{Acknowledgement}

\newtheorem{conjecture}{Conjecture}
\newtheorem{corollary}[theorem]{Corollary}

\newtheorem{definition}[theorem]{Definition}
\newtheorem{lemma}[theorem]{Lemma}

\newtheorem{proposition}[theorem]{Proposition}
\newtheorem*{thmannounce}{Theorem}
\newtheorem{unnumthmrmk}[theorem]{Theorem}
\theoremstyle{remark}
\newtheorem{remark}[theorem]{Remark}

\newtheorem{philosophy}[theorem]{Philosophy}
\newtheorem{example}[theorem]{Example}

\numberwithin{equation}{section}
\begin{document}
\title[$K$-theory of semi-linear endomorphisms]{$K$-theory of semi-linear endomorphisms via the Riemann--Hilbert correspondence}
\author{Oliver Braunling}
\address{Albert Ludwig University of Freiburg, Eckerstra\ss e 1, D-79104 Freiburg, Germany}
\urladdr{http://www.braunling.org/}
\thanks{The author has been supported by the GK1821 \textquotedblleft Cohomological
Methods in Geometry\textquotedblright.}

\begin{abstract}
Grayson, developing ideas of Quillen, has made computations of the $K$-theory
of `semi-linear endomorphisms'. In the present text we develop a technique to
compute these groups in the case of Frobenius semi-linear actions. The main
idea is to interpret the semi-linear modules as crystals and use a positive
characteristic version of the\ Riemann--Hilbert correspondence. We also
compute the $K$-theory of the category of \'{e}tale constructible $p$-torsion sheaves.

\end{abstract}
\maketitle

If $R$ is a commutative regular $\mathbf{F}_{q}$-algebra, we may define the
Frobenius skew ring%
\begin{equation}
R[F]:=\left.  R\{F\}\right.  /\left.  \left\langle x^{q}F-Fx\mid x\in
R\right\rangle \text{.}\right. \label{lcin1}%
\end{equation}
This is a non-commutative ring. If we ignore the meaning of $q:=\#\mathbf{F}%
_{q}$ and pretended \textquotedblleft$q=1$\textquotedblright\ formally, this
definition would output the ordinary polynomial ring. However, one of the most
fundamental properties of $K$-theory is its $\mathbf{A}^{1}$-invariance, i.e.%
\begin{equation}
K(R)\overset{\sim}{\longrightarrow}K(R[T])\text{,}\label{lcin2}%
\end{equation}
so one is tempted to hope that this remains true when $q\neq1$. Our first
result is that this is indeed the case:

\begin{thmannounce}
Let $X/\mathbf{F}_{q}$ be a smooth separated scheme. Then there is an
equivalence in $K$-theory%
\[
K(X)\overset{\sim}{\longrightarrow}K(\mathsf{Coh}_{\mathcal{O}_{X}%
[F]}(X))\text{.}%
\]
Here $\mathsf{Coh}_{\mathcal{O}_{X}[F]}(X)$ refers to coherent right
$\mathcal{O}_{X}[F]$-module sheaves.
\end{thmannounce}

See Theorem \ref{thm_TwistedA1InvarianceOfKTheory}. The proof will be much
like the one for \textquotedblleft$q=1$\textquotedblright, but with a critical
technical complication: $\mathcal{O}_{X}[F]$ is neither left nor right
Noetherian when $q\neq1$. But this can be managed, thanks to Emerton's insight
that they remain coherent rings \cite{EmertonCoh}, i.e. all finitely generated
ideals are automatically finitely presented. He shows left coherence, and we
complement this with right coherence in the present text.\medskip

Alternatively, we could look at those (right) $\mathcal{O}_{X}[F]$-module
sheaves which are additionally coherent as $\mathcal{O}_{X}$-module sheaves.
This yields the category of coherent Cartier modules $\mathsf{CohCart}(X)$ of
Blickle and B\"{o}ckle \cite{MR2863904}. Its $K$-theory sees additional
arithmetic information:

\begin{thmannounce}
Suppose $X/\mathbf{F}_{q}$ is a smooth separated scheme. Then there is a long
exact sequence in $K$-theory%
\[
\cdots\rightarrow K_{m}(X)\rightarrow K_{m}(\mathsf{CohCart}(X))\rightarrow
K_{m}(\mathsf{\acute{E}t}_{c}(X,\mathbf{F}_{q}))\rightarrow\cdots\text{.}%
\]

\end{thmannounce}

Here $\mathsf{\acute{E}t}_{c}(X,\mathbf{F}_{q})$ denotes the abelian category
of constructible \'{e}tale sheaves with $\mathbf{F}_{q}$-coefficients. This
computation is based on the positive characteristic version of the
Riemann--Hilbert correspondence of Emerton--Kisin \cite{MR2071510} and
Blickle--B\"{o}ckle \cite{MR2863904}.\medskip

We also describe $K(\mathsf{\acute{E}t}_{c}(X,\mathbf{F}_{q}))$ to some
extent. So far, the only computation regarding the $K$-theory of this category
that I've seen in the literature is due to Taelman \cite{taelmanwoodshole}. He
develops a function-sheaf correspondence for these \'{e}tale sheaves. To each
sheaf and $\mathbf{F}_{q}$-rational point, one pulls back the sheaf to this
point and takes the trace of the Frobenius action. This assigns a value to
each rational point. Taelman shows that this construction factors over the
$K_{0}$-group of $\mathsf{\acute{E}t}_{c}(X,\mathbf{F}_{q})$, and thus gives
rise to a short exact sequence%
\begin{equation}
0\rightarrow K_{0}(\mathsf{\acute{E}t}_{c}(X,\mathbf{F}_{q}%
))^{\operatorname*{Tr}=0}\rightarrow K_{0}(\mathsf{\acute{E}t}_{c}%
(X,\mathbf{F}_{q}))\rightarrow\bigoplus_{x\in X(\mathbf{F}_{q})}\mathbf{F}%
_{q}\rightarrow0\label{lTae}%
\end{equation}
and he gives explicit generators for the term on the left-hand side. We
augment this computation of the $K_{0}$-group as follows:

\begin{thmannounce}
Suppose $X/\mathbf{F}_{q}$ is a smooth scheme. Then%
\[
K_{m}(\mathsf{\acute{E}t}_{c}(X,\mathbf{F}_{q}))=\left\{
\begin{array}
[c]{ll}%
\text{\emph{prime-to-}}p\text{\emph{\ torsion}} & \text{for }m=2i+1>0\text{,}%
\\
0 & \text{for }m=2i\text{, }i>0\text{,}\\
\bigoplus\mathbf{Z} & \text{for }m=0\text{,}%
\end{array}
\right.
\]
where the direct sum in the last row runs over all simple Cartier crystals on
$X$ (or equivalently simple perverse \'{e}tale $\mathbf{F}_{q}$-sheaves).
Among them, there is a canonical set of $\#X(\mathbf{F}_{q})$ generators which
surject on the right-hand side in Sequence \ref{lTae}, while the remaining
generators all map to zero under the function-sheaf correspondence.
\end{thmannounce}

See Theorem \ref{thm_Summary_KThyOfCartCrystals}. For this result, we first
switch to the perverse $t$-structure by Neeman's theorem of the heart, and
then to Cartier crystals under the Riemann--Hilbert correspondence and work in
the latter context. There, the computation crucially uses the strong
finiteness properties proven by Blickle and\ B\"{o}ckle \cite{MR2863904}, and
Quillen's computation of the $K$-theory of finite fields.\medskip

There is also a completely different angle from which to look at this text:
Instead of the $K$-theory of a category itself, one may look at the category
of pairs $(X,\alpha)$, where $X$ is an object and $\alpha$ acts on it via (1)
an endomorphism, (2) or an automorphism, or (3) a semi-linear endomorphism of
$X$. Grayson has a series of articles investigating these cases
\cite{MR0480699}, \cite{MR535839}, \cite{MR929766}. The case of automorphisms
has gained some fame for its r\^{o}le in the motivic Atiyah-Hirzebruch
spectral sequence \cite{MR1340843}, and recently there has been some renewed
interest in the cases of endo- and automorphisms, e.g. \cite{MR3167624},
\cite{MR3430369}.

The case of semi-linear actions seems to be studied less. Basically, an
example of Quillen in \cite{MR0338129} and Grayson's paper \cite{MR929766}
seem to be the only ones computing higher $K$-groups (beyond $K_{1}$) in the
case of semi-linear actions. Seen from this angle, the above theorems provide
a large supply of further computations in the special case when the
semi-linear action comes from the\ Frobenius. Grayson's paper used a
\textquotedblleft Frobenius $\mathbf{P}^{1}$\textquotedblright, which inspires
our use of a twisted affine line, however the projective line seems to be
definable only for rings for which the Frobenius is an automorphism.\medskip

In this direction, the technical advance of the present text is that we can
handle rings where the Frobenius is \textit{not} surjective. This is the key
point which allows us to handle rings and varieties over $\mathbf{F}_{q}$ of
dimension $\geq1$.\bigskip

Further results:\medskip

This paper has a precursor in a computation in Quillen's \textquotedblleft
Higher Algebraic $K$-Theory I\textquotedblright. Quillen considers the
$K$-theory of semi-linear endomorphisms of the single point
$X:=\operatorname*{Spec}(\mathbf{F}_{p}^{\operatorname*{sep}})$. Based on it,
he defines a skew field $D$ and computes its $K$-theory. The meaning of these
$K$-theory classes remained mysterious. We try to elucidate his computation by
adapting it to schemes over $\mathbf{F}_{p}$ in \S \ref{sect_QuillenRevisited}%
. We explain that it is impossible to define an analogous skew field in this
generality, but there is an abelian category \textquotedblleft$\mathsf{QD}%
$\textquotedblright, which imitates the behaviour of modules over the
non-existing $D$. Assuming Parshin's conjecture, we can compute its rational
$K$-theory and find%
\[%
\begin{tabular}
[c]{l|l|l}%
Quillen's $K_{m}(D)_{\mathbf{Q}}$ & Generalized $K_{m}(\mathsf{QD}%
(X))_{\mathbf{Q}}$ & $m$\\\hline
$\mathbf{Q}$ & $\mathbf{Q}$ & $0$\\
$\mathbf{Q}\oplus\mathbf{Q}$ & $\mathbf{Q}\oplus K_{0}(\mathsf{\acute{E}t}%
_{c}(X,\mathbf{F}_{q}))_{\mathbf{Q}}$ & $1$\\
$0$ & $0$ & $\geq2$.\\
\multicolumn{1}{r|}{$\qquad$for $X=\operatorname*{Spec}(\mathbf{F}%
_{p}^{\operatorname*{sep}})$} & $\qquad$for $X$ smooth, & \\
\multicolumn{1}{r|}{} & $\qquad$projective over $\mathbf{F}_{q}$ &
\end{tabular}
\]
One can see how the individual summands in Quillen's computation generalize,
which might be a first step in understanding the bigger picture behind
Quillen's computation.

\section{Preparations}

\textit{Conventions:} A ring $R$ denotes an associative unital algebra, not
necessarily commutative. Ring morphisms are always supposed to preserve the
unit. If $R$ is a ring, $\mathsf{Mod}_{fg}(R)$ resp. $\mathsf{Mod}_{fp}(R)$
denote the category of finitely generated resp. finitely presented right
$R$-modules. We write $\mathsf{P}_{f}(R)$ to denote the exact category of
finitely generated projective right $R$-modules. We call a ring \emph{right
regular} if every finitely presented right module has finite projective
dimension. Unlike most commutative algebra texts, we do not demand regular
rings to be Noetherian. In particular, left and right global dimensions may
differ.\medskip

We pick once and for all a prime number $p$ and a prime power $q=p^{r}$ with
$r\geq1$. Let $R$ be a ring. We write $R\{X\}$ for the free associative ring
in a non-commuting variable $X$. We write $R[X]$ to denote the polynomial ring
over $R$, i.e. in this case the variable $X$ commutes with all elements of the
ring. In other words,%
\begin{equation}
R[X]=R\{X\}/\left\langle rX-Xr\mid r\in R\right\rangle \text{.}\label{lttt0}%
\end{equation}
We will only deviate from this notation in one special case: Suppose $R$
denotes a commutative $\mathbf{F}_{q}$-algebra. Define the non-commutative
\emph{Frobenius skew ring}%
\begin{equation}
R[F]:=\frac{R\{F\}}{\left\langle x^{q}F-Fx\mid x\in R\right\rangle }%
\text{.}\label{ltttta1}%
\end{equation}
So we reserve the special letter \textquotedblleft$F$\textquotedblright\ for
this definition differing from the one in line \ref{lttt0}. This is fairly
common practice in the literature.

\subsection{Strategy}

Before we start with precise arguments, let us just explain what we want to
do: The ring inclusion $R\subset R[F]$ induces a morphism in $K$-theory%
\[
K(R)\rightarrow K(R[F])
\]
and we would like to show that this functor induces an equivalence in
$K$-theory. The idea is to imitate Quillen's proof of $\mathbf{A}^{1}%
$-invariance of $K$-theory for regular $R$: He proves the equivalence
$K(R)\overset{\sim}{\rightarrow}K(R[T])$, where $R[T]$ is the ordinary
polynomial ring and $R$ a regular Noetherian ring. The key tool in Quillen's
proof is the following result:

\begin{theorem}
[{\cite[Theorem 7]{MR0338129}}]\label{thm_Quillen}Let $A$ be an increasingly
filtered ring%
\[
A=\bigcup_{s\geq0}A^{\leq s}%
\]
such that $A^{\leq s}\cdot A^{\leq t}\subseteq A^{\leq s+t}$. Suppose the
associated graded $\operatorname*{Gr}A:=\bigoplus_{s\geq0}A^{\leq s}/A^{\leq
s-1}$ (with the tacit understanding that $A^{\leq-1}:=0$) is right Noetherian
and has finite $\operatorname*{Tor}$-dimension as a right module over
$A^{\leq0}$. Moreover, assume that $A^{\leq0}$ has finite $\operatorname*{Tor}%
$-dimension as a right $\operatorname*{Gr}A$-module. Then the ring
homomorphism $A^{\leq0}\hookrightarrow A$ induces an equivalence in $K$-theory%
\[
K(\mathsf{Mod}_{fg}A^{\leq0})\overset{\sim}{\longrightarrow}K(\mathsf{Mod}%
_{fg}A)\text{.}%
\]

\end{theorem}

Quillen then combines this with a comparison result between the $K$-theory of
finitely generated projective right $A$-modules and finitely generated right
$A$-modules, namely%

\begin{equation}
K(A^{\leq0})\overset{\sim}{\longrightarrow}K(\mathsf{Mod}_{fg}A^{\leq0}%
)\qquad\text{and}\qquad K(A)\overset{\sim}{\longrightarrow}K(\mathsf{Mod}%
_{fg}A)\text{,}\label{ltad1}%
\end{equation}
which requires $A_{0}$ and $A$ to be right Noetherian and right regular. So,
this is the plan. However, we cannot just follow this strategy, because it
collapses at a number of places for the Frobenius skew ring $R[F]$:

\begin{enumerate}
\item The ring $R[F]$ is practically never right Noetherian. In particular,
the category of finitely generated right $R[F]$-modules, $\mathsf{Mod}%
_{fg}R[F]$, a priori need not be an abelian category. For the moment, this is
not too bad, as it still is an exact category and thus has a notion of $K$-theory.

\item The ring $R[F]$ is indeed filtered by $R[F]^{\leq d}:=\{\sum_{i=0}%
^{d}r_{i}F^{i}\}$. One easily computes that%
\[
\operatorname*{Gr}R[F]\simeq R[F]\text{,}%
\]
i.e. the associated graded is isomorphic to the original ring. However, since
$R[F]$ is rarely right Noetherian, this means that $\operatorname*{Gr}R[F]$
will fail to be right Noetherian, too. So it cannot satisfy the assumptions of
Theorem \ref{thm_Quillen}.

\item The analogues of the comparison results in line \ref{ltad1} require
right regularity. This is in fact a rather recent result of Linquan Ma
\cite[Theorem 3.2]{MR3160413}.

\item We solve the non-Noetherian problem by proving right coherence of $R[F]
$ under suitable conditions (based on a method of Emerton, \cite{EmertonCoh}).
But this still does not quite suffice because if a ring $A$ is right coherent,
its polynomial ring $A[T]$ need not be right coherent as well (by an example
due to Soublin \cite[\S 5]{MR0260799}), and we will have no better tool than
proving the relevant right coherence statements by hand. Based on this, we can
then use a strengthening of Quillen's theorem, Theorem \ref{thm_Quillen}, for
right coherent rings, due to Gersten \cite{MR0396671}.
\end{enumerate}

\section{Ring-theoretic properties of the Frobenius skew ring}

\subsection{Generalities}

Let us collect a few properties of the Frobenius skew ring, defined as in line
\ref{ltttta1}: Suppose $R$ is a commutative $\mathbf{F}_{q}$-algebra. Every
element in $R[F]$ has a unique presentation as a left polynomial%
\[
\alpha=\sum_{i=0}^{d}r_{i}F^{i}\qquad\text{with}\qquad r_{i}\in R\text{.}%
\]
Thus, $R[F]$ is a free left $R$-module. It is also a right $R$-module because
of $rF\cdot s=rs^{q}F$. However, it need not be free as a right module, nor
will it in general be possible to represent elements as right polynomials
$\sum F^{i}r_{i}$.

If $f:R\rightarrow S$ is a ring morphism of commutative $\mathbf{F}_{q}%
$-algebras, the defining relation in line \ref{ltttta1} is preserved and one
obtains an induced morphism $R[F]\rightarrow S[F]$.

We would like to speak about finitely generated left or right $R[F]$-modules.
However, we directly run into problems since $R[F]$ is only very rarely left
or right Notherian. This is a problem because for a general ring, its category
of finitely generated modules will not even be an abelian category.

The question of being Noetherian was settled in full generality by Yuji
Yoshino \cite{MR1271618}.

\begin{theorem}
[Yoshino]\label{thm_Yoshino_CharacterizationNoetherian}Suppose $R$ is a
Noetherian commutative $\mathbf{F}_{p}$-algebra.

\begin{enumerate}
\item Then $R[F]$ is left Noetherian iff $R$ is a direct product of finitely
many fields.

\item Then $R[F]$ is right Noetherian iff $R$ is Artinian and all closed
points in $\operatorname*{Spec}R$ have perfect residue fields.
\end{enumerate}
\end{theorem}

See \cite[Theorem 1.3]{MR1271618}. The proof of this general version is quite
involved. In the classical case of $R=k$ a perfect field, one can intepret
$R[F]$ as a twisted polynomial ring. This case has textbook treatments, e.g.
\cite[2.9, Theorem, (iv)]{MR1811901}.

As we can see, we need a workaround handling the lack of Noetherian properties
since the above cases are far too special to be useful. They all have
$\operatorname*{Spec}R$ zero-dimensional. We shall use the formalism of
coherent rings. We recall all necessary foundations:

Suppose $A$ is a ring. A right $A$-module is called \emph{coherent} if (1) it
is finitely generated, and (2) every finitely generated right submodule is
finitely presented. For every short exact sequence of right $A$-modules,%
\[
0\rightarrow K\rightarrow L\rightarrow M\rightarrow0
\]
all modules are coherent as soon as any two of them are coherent (see
Soublin's survey \cite{MR0260799} for this and related properties).

A ring $A$ is called \emph{right coherent} if every finitely generated right
ideal is also finitely presented. Equivalently, every finitely generated right
submodule of a free right module is finitely presented. Clearly right
Noetherian rings are also right coherent.

\begin{example}
[{\cite[\S 5, Corollaire]{MR0260799}}]The simplest example of a ring which is
not right Noetherian, but right coherent, is the polynomial ring in countably
many variables over a commutative Noetherian ring $R$, i.e. $A:=R[X_{1}%
,X_{2},\ldots]$.
\end{example}

\begin{proposition}
\label{prop_CharacterizationsOfRightCoherentRings}For a ring $A$, the
following are equivalent:

\begin{enumerate}
\item $A$ is right coherent.

\item $A$ is a coherent right $A$-module over itself.

\item The category of finitely presented right $A$-modules is abelian.
\end{enumerate}
\end{proposition}

We shall need the following facts:

\begin{proposition}
\label{prop_PropsAroundCoherence}Suppose $A$ is a right coherent ring.

\begin{enumerate}
\item Then every finitely generated projective right $A$-module is coherent.
(\cite[\S 3, Prop. 9]{MR0260799})

\item If $I$ is a two-sided ideal in $A$, which is finitely generated as a
right $A$-module, then the quotient ring $A/I$ is also right coherent.
(\cite[\S 4, Cor. 1]{MR0260799})
\end{enumerate}
\end{proposition}

The latter fact implies that if the polynomial ring $A[T]$ is right coherent,
so is $A$. The converse direction is known to be false:

\begin{example}
[Soublin]There exists a commutative coherent ring $A$ such that $A[T]$ is not
coherent. See \cite[\S 5, Prop. 18]{MR0260799}.
\end{example}

\subsection{Right stable coherence\label{sect_CoherenceTheorems}}

So, the notion of a coherent ring is not $\mathbf{A}^{1}$-invariant. Inspired
by this, Gersten has introduced the following notion:

\begin{definition}
[{\cite[Definition 1.2]{MR0396671}}]A ring $A$ is called \emph{right stably
coherent} (a.k.a. \textquotedblleft right super-coherent\textquotedblright) if
for any index set $I$ the multi-variable polynomial ring $A[X_{i}]_{i\in I}$
is right coherent.\footnote{Gersten has called such rings `right
super-coherent'. However, the majority of the literature prefers to call them
`right stably coherent'. Indeed, Aschenbrenner has introduced another concept
called super-coherence, which is related, but different from stable coherence,
and in particular different from Gersten's usage.}
\end{definition}

Now, Matt Emerton has shown that the Frobenius skew ring is left coherent.

\begin{theorem}
[Emerton \cite{EmertonCoh}]\label{marker_EmertonLeftCoherence}Suppose

\begin{itemize}
\item $R$ is a Noetherian commutative $\mathbf{F}_{q}$-algebra, and

\item $R$ is $F$-flat (e.g. if $R$ is regular\footnote{By\ Kunz Theorem
\cite[Theorem 2.1]{MR0252389}, if $R$ is a reduced Noetherian commutative
$\mathbf{F}_{q}$-algebra, it is $F$-flat if and only if it is regular.}).
\end{itemize}

Then the ring $R[F]$ is left coherent.
\end{theorem}

This is the main result of \cite{EmertonCoh}. We will now prove the right
analogue of Emerton's theorem. The assumption of $F$-flatness will need to be
replaced by $F$-finitness, as it turns out. Apart from the minor modifications
circling around this, the proof will be a mirror image of Emerton's proof,
albeit augmented with additional commuting polynomial variables.\medskip

Because only right stable coherence will be useful for us later, we shall need
to prove right coherence for polynomial rings over $R[F]$. Let us set up the
notation. We have%
\begin{equation}
R[F][T]:=\left\{  \left.  \sum r_{i,\ell}F^{i}T^{\ell}\right\vert r_{i,\ell
}\in R\text{, all but finitely many }r_{i,\ell}\text{ are zero}\right\}
\label{lah0}%
\end{equation}
and recall that by the construction of this ring, we have the relations%
\begin{equation}
rF^{i}T^{\ell}\cdot s=rs^{q^{i}}F^{i}T^{\ell}\qquad T\cdot rF^{i}T^{\ell
}=rF^{i}T^{\ell+1}\qquad TF=FT\label{lah1}%
\end{equation}
for all $r,s\in R$. In particular, $T$ commutes with all other terms.

\begin{example}
The ring $R[F][T]$ is different from $R[T][F]$, because the latter ring
satisfies the relation $FT=T^{q}F$ instead.
\end{example}

\begin{remark}
\label{rmk_MultivariableRing}All of the following considerations also work
verbatim for multi-variable polynomial rings over $R[F]$, i.e. $R[F][T_{1}%
,\ldots,T_{m}]$ for any integer $m\geq0$. In fact, just read the definition in
line \ref{lah0} as allowing a multi-index $\underline{\ell}=(\ell_{1}%
,\ldots,\ell_{m})$ for $\ell$, and the same finiteness condition on the
coefficients $r_{i,\ell}$.\ For the sake of legibility, we proceed by writing
a single $T$.
\end{remark}

\begin{definition}
[Right twisting the action]If $M$ is a right $R[T]$-module, write
$\widetilde{M}$ for the right $R[T]$-module with the same elements as $M$, but
with the twisted right $R[T]$-module structure%
\[
m\underset{\widetilde{M}}{\cdot}sT^{j}:=m\underset{M}{\cdot}s^{q}T^{j}\text{.}%
\]
When convenient, we may even use this notation for right $R[F][T]$-modules,
meaning%
\[
m\underset{\widetilde{M}}{\cdot}sF^{i}T^{j}:=m\underset{M}{\cdot}s^{q}%
F^{i}T^{j}\text{.}%
\]
\newline\textsl{If} there is also a left $R[T]$-module structure on $M$, we do
not change it.
\end{definition}

\begin{remark}
\label{remark_on_notation}One can also denote this by $M^{(1)}$ or by
$F_{\ast}M$ for $F:R\rightarrow R$ being the map $r\mapsto r^{q}$. The latter
might actually be the most official notation, but I am very hesitant to use
it. Firstly, if $M$ carries both a left and right $F$ action, the notation
$F_{\ast}M$ is very ambiguous. Secondly, we will switch a lot between schemes
and rings and then the Frobenius morphism goes in opposite directions
depending on which viewpoint we use. Both feels all too prone to lead to confusion.
\end{remark}

\begin{lemma}
\label{lem_fingen1}Suppose $R$ is a commutative $\mathbf{F}_{q}$-algebra. The
following are equivalent:

\begin{enumerate}
\item $R$ is a finitely generated (right) $R^{q}$-module.

\item $\widetilde{R}$ is a finitely generated (right) $R$-module.

\item For every $d\geq0$, $R$ is a finitely generated (right) $R^{q^{d}}$-module.
\end{enumerate}
\end{lemma}

The proof is straightforward. The ring $R$ is called $F$\emph{-finite }if
these conditions are met. This property is very common. Perfect fields are $F
$-finite, and $F$-finiteness is closed under taking finitely generated algebra
extensions, localizations and homomorphic images.

\begin{lemma}
\label{lem_fingenA}Suppose $R$ is a commutative $F$-finite $\mathbf{F}_{q}$-algebra.

\begin{enumerate}
\item If $M$ is a finitely generated right $R[F][T]$-module, then so is
$\widetilde{M}$.

\item If $R$ is additionally Noetherian: If $M$ is a finitely presented right
$R[F][T]$-module, then so is $\widetilde{M}$.
\end{enumerate}
\end{lemma}

\begin{proof}
This follows literally from $F$ being a finite morphism. \textit{(1)}
Concretely, for finite generation, if $\left\{  b_{i}\right\}  _{i\in I}$ is a
finite set of right generators for $M$, and $\left\{  \rho_{s}\right\}  _{s\in
J}$ a finite set of generators of $R$ as a right $R^{q}$-module, then every
element $m\in\widetilde{M}$ can be written as a finite sum $m=\sum b_{i}f_{i}$
for suitable $f_{i}\in R[F][T]$, and each of these as $f_{i}=\sum
_{j,k}r_{i,j,k}F^{j}T^{k}=\sum_{j,k,s}\rho_{s}r_{i,j,k,s}^{q}F^{j}T^{k}$.
Hence,%
\[
m=\sum_{i,s}b_{i}\rho_{s}\underset{M}{\cdot}\sum_{j,k}r_{i,j,k,s}^{q}%
F^{j}T^{k}=\sum_{i,s}b_{i}\rho_{s}\underset{\widetilde{M}}{\cdot}\sum
_{j,k}r_{i,j,k,s}F^{j}T^{k}\text{,}%
\]
proving that $\{b_{i}\rho_{s}\}_{i\in I,s\in J}$ is a finite set of generators
as a right $R[F][T]$-module. \textit{(2)} As $\widetilde{R}$ is a finitely
generated $R$-module, and $R$ is Noetherian, $\widetilde{R}$ is also a
finitely presented $R$-module, i.e. there are $n,m$ such that $R^{\oplus
n}\rightarrow R^{\oplus m}\rightarrow\widetilde{R}\rightarrow0$ is exact.
Since $R[F][T]$ is a free (thus flat) left $R$-module, tensoring it from the
right yields%
\[
R[F][T]^{\oplus n}\rightarrow R[F][T]^{\oplus m}\rightarrow\widetilde
{R}\otimes_{R}R[F][T]\rightarrow0
\]
and the term on the right equals $\widetilde{R}[F][T]=\widetilde{R[F][T]}$
(recall that by our definition of $\widetilde{(-)}$, only the right $R$-module
structure changes, e.g. right multiplication by $r$ becomes right
multiplication by $r^{q}$, but right multiplication by $T$ remains right
multiplication by $T$). Thus, $\widetilde{R[F][T]}$ is a finitely presented
right $R[F][T]$-module. Hence, if%
\[
R[F][T]^{\oplus n}\rightarrow R[F][T]^{\oplus m}\rightarrow M\rightarrow0
\]
is a finite presentation of a right $R[F][T]$-module $M$, we get%
\[
\widetilde{R[F][T]}^{\oplus n}\rightarrow\widetilde{R[F][T]}^{\oplus
m}\rightarrow\widetilde{M}\rightarrow0\text{,}%
\]
presenting $\widetilde{M}$ as a quotient of a finitely presented module by a
finitely generated one, implying that it is finitely presented itself.
\end{proof}

\begin{corollary}
\label{cor_fingenA}Suppose $R$ is a commutative $F$-finite Noetherian
$\mathbf{F}_{q}$-algebra. If $M$ is a finitely generated (resp. presented)
$R$-module, then so is $\widetilde{M}$.
\end{corollary}

For a short exact sequence of right $R[F][T]$-modules, one gets the diagram of
right $R$-modules%
\[%
\begin{array}
[c]{rccccclc}%
0\rightarrow & \widetilde{A} & \rightarrow & \widetilde{B} & \rightarrow &
\widetilde{C} & \rightarrow0 & \\
& \downarrow &  & \downarrow &  & \downarrow &  & \cdot F\\
0\rightarrow & A & \rightarrow & B & \rightarrow & C & \rightarrow0\text{.} &
\end{array}
\]
The downward arrows denote the \textit{right} action by $F$. Thanks to the
tilde twist in the upper row, this is a right $R$-module homomorphism (had we
not twisted the right action in the top row, this would only be an abelian
group homomorphism).

The following lemmata are right analogues of corresponding statements in
\cite{EmertonCoh}, suitably stabilized along additional commuting polynomial variables.

\begin{lemma}
\label{lem_fingen2}Suppose $R$ is a commutative $F$-finite $\mathbf{F}_{q}%
$-algebra. Then for all $d\geq0$ the following set is a finitely generated
(right) $R[T]$-module:
\end{lemma}%

\[
Z^{\leq d}:=\left\{  \left.  \sum_{i=0}^{d}\sum_{\ell}r_{i,\ell}F^{i}T^{\ell
}\right\vert r_{i,\ell}\in R\text{, all but finitely many }r_{i,\ell}\text{
are zero}\right\}
\]

\begin{proof}
First of all, we see that it is a right $R[T]$-module by%
\[
r_{i,\ell}F^{i}T^{\ell}\cdot s=r_{i,\ell}s^{q^{i}}F^{i}T^{\ell}\qquad
\text{and}\qquad r_{i,\ell}F^{i}T^{\ell}\cdot T=r_{i,\ell}F^{i}T^{\ell
+1}\text{.}%
\]
By the previous lemma and our assumption that $R$ be $F$-finite, $R$ is a
finitely generated (right) $R^{q^{d}}$-module. Thus, every element $r\in R$
can be written in the shape $r=\sum_{s\in\mathcal{B}_{d}}s\rho_{s}^{q^{d}}$,
where $\mathcal{B}_{d}$ is a finite set of right generators for $R$ and
$\rho_{s}$ the coefficients as a right $R^{q^{d}}$-module (i.e. they act as
$\rho_{s}^{q^{d}}$ with respect to the ordinary right $R$-module structure).
Hence, every element in $Z^{\leq d}$ has the shape%
\[
z=\sum_{\ell}\sum_{i=0}^{d}r_{i,\ell}F^{i}T^{\ell}=\sum_{\ell}\sum_{i}%
^{d}\left(  \sum_{s\in\mathcal{B}_{i}}s\rho_{i,\ell,s}^{q^{i}}\right)
F^{i}T^{\ell}%
\]
and this equals%
\[
=\sum_{i=0}^{d}\sum_{s\in\mathcal{B}_{i}}sF^{i}\sum_{\ell}\rho_{i,\ell
,s}T^{\ell}\text{.}%
\]
We see that $Z^{\leq d}$ is spanned as a right $R[T]$-module by $\left\langle
sF^{i}\right\rangle _{0\leq i\leq d,s\in\mathcal{B}_{i}}$. This is a finite
set since each $\mathcal{B}_{i}$ is finite and $i$ runs through finitely many
values only.
\end{proof}

\begin{lemma}
\label{lem_fingen3}Suppose $R$ is a Noetherian $F$-finite commutative
$\mathbf{F}_{q}$-algebra. Then for every right ideal $I$ in $R[F][T]$, each%
\[
I^{\leq d}:=I\cap Z^{\leq d}%
\]
is a finitely generated (right) $R[T]$-submodule of $I$ and their union
$\bigcup_{d\geq0}I^{\leq d}$ is all of $I$.
\end{lemma}

\begin{proof}
For any $d\geq0$, we have%
\[
I^{\leq d}=\left\{  \left.  \sum_{\ell}\sum_{i=0}^{d}r_{i,\ell}F^{i}T^{\ell
}\right\vert r_{i,\ell}\in R\text{, all but finitely many }r_{i,\ell}\text{
are zero}\right\}  \cap I\text{.}%
\]
Since both $I$ and $Z^{\leq d}$ are right $R[T]$-modules, so is their
intersection $I^{\leq d}$. By Lemma \ref{lem_fingen2} the right module
$Z^{\leq d}$ is a finitely generated (right) $R[T]$-module. Since $R$ is
Noetherian, the (commutative) polynomial ring $R[T]$ is Noetherian. It follows
that $Z^{\leq d}$ is a Noetherian right $R[T]$-module and thus its
$R[T]$-submodule $I^{\leq d}=I\cap Z^{\leq d}\subseteq Z^{\leq d}$ is also a
finitely generated right $R[T]$-module. We also have $I=\bigcup_{d}I^{\leq d}
$, because every element in $I$ has the form $r_{0,0}+\cdots+r_{d,d}F^{d}%
T^{d}$ for some sufficiently large $d$, so it will lie in $Z^{\leq d}$ and $I$ simultaneously.
\end{proof}

\begin{lemma}
[Emerton's Key Lemma, right analogue]\label{lem_rightemertonkeylemma}Suppose

\begin{itemize}
\item $R$ is a Noetherian commutative $\mathbf{F}_{q}$-algebra,

\item $I$ is a right ideal in $R[F][T]$,

\item $R$ is $F$-finite.
\end{itemize}

Then $I$ is a finitely generated right ideal in $R[F][T]$ if and only if the
right $R[T]$-module $\operatorname*{coker}(\widetilde{I}\overset{\cdot
F}{\longrightarrow}I)$ is finitely generated.
\end{lemma}

\begin{proof}
\textit{(1)} Suppose $I$ is a finitely generated right ideal in $R[F][T]$.
Then there exists some $n\geq0$ such that the diagram of right $R[T]$-modules%
\[%
\begin{array}
[c]{rccccclc}%
0\rightarrow & \widetilde{K} & \rightarrow & \widetilde{R[F][T]^{n}} &
\rightarrow & \widetilde{I} & \rightarrow0 & \\
& \downarrow &  & \downarrow &  & \downarrow &  & \cdot F\\
0\rightarrow & K & \rightarrow & R[F][T]^{n} & \rightarrow & I &
\rightarrow0 &
\end{array}
\]
commutes. Recall that the right action twist means that $\tilde{m}%
\underset{\widetilde{M}}{\cdot}sT^{j}:=\tilde{m}\underset{M}{\cdot}s^{q}T^{j}$
in the top row. The snake lemma gives us a surjection%
\[
(R[F][T]/R[F][T]F)^{n}\twoheadrightarrow\operatorname*{coker}(\widetilde
{I}\overset{\cdot F}{\rightarrow}I)
\]
(\footnote{Note that the image of $\widetilde{R[F][T]}$ under right
multiplication by $F $ is $R[F][T]F$, and not $\widetilde{R[F][T]}F$.}). There
is an obvious right $R[T]$-module isomorphism\footnote{Actually this is even a
right $R[F][T]$-module isomorphism, where $F$ acts as $r\cdot F=0$ for all
$r\in R[T]$.}%
\begin{align}
R[F][T]/R[F][T]F  & \longrightarrow R[T]\label{lah2}\\
\sum r_{i,\ell}F^{i}T^{\ell}  & \longmapsto\sum r_{0,\ell}T^{\ell}%
\text{.}\nonumber
\end{align}
Note that for $s\in R$, we get%
\[
r_{i,\ell}F^{i}T^{\ell}\cdot s=r_{i,\ell}s^{q^{i}}F^{i}T^{\ell}\longmapsto
r_{0,\ell}sT^{\ell}=r_{0,\ell}T^{\ell}\cdot s\text{.}%
\]
We conclude that the cokernel is a finitely generated right $R[T]$%
-module.\newline\textit{(2)} Suppose $\operatorname*{coker}(\widetilde
{I}\overset{\cdot F}{\rightarrow}I)$ is a finitely generated right
$R[T]$-module. We consider the maps of right $R[T]$-modules%
\begin{equation}
I^{\leq d}\longrightarrow\operatorname*{coker}(\widetilde{I}\overset{\cdot
F}{\rightarrow}I)\text{.}\label{lttt1}%
\end{equation}
The images, along $d\rightarrow+\infty$, form an ascending chain of right
$R[T]$-submodules. As we had assumed that the right-hand side module was right
finitely generated over $R[T]$ and thus Noetherian (since $R$ and thus $R[T]$
are Noetherian), this chain must become stationary. On the other hand, taking
the union over all $d$, we get the entire image of $I$, but $I$ surjects onto
the cokernel. Thus, we learn that there exists some $d_{0}$ such that the map
in line \ref{lttt1} is surjective for all $d\geq d_{0}$. Next, we claim that%
\begin{equation}
I^{\leq d}\subseteq I^{\leq d_{0}}+IF\text{.}\label{lttt2}%
\end{equation}
This is clear: By line \ref{lttt1} every element in $I/IF$ comes from $I^{\leq
d_{0}}$, so in $I$ it lies in $I^{\leq d_{0}}+IF$. Moreover, we claim that%
\begin{equation}
IF\cap I^{\leq d}=I^{\leq d-1}F\text{.}\label{lttt3}%
\end{equation}
(Proof: Right to left: Every element in $I^{\leq d-1}F$ clearly lies in $IF$,
as well as $I^{\leq d}$. For the converse direction: Suppose $\alpha\in IF\cap
I^{\leq d}$. Then we have a presentation%
\[
\alpha=\sum_{\ell}\sum_{i=0}^{d}r_{i,\ell}F^{i}T^{\ell}=\beta F\qquad
\text{with}\qquad\beta=\sum_{\ell}\sum_{i=0}^{\infty}b_{i,\ell}F^{i}T^{\ell
}\in I\text{.}%
\]
Thus, by the uniqueness of presentations as left polynomials in $F$, we get%
\[
\sum_{\ell}\sum_{i=0}^{d}r_{i,\ell}F^{i}T^{\ell}=\sum_{\ell}\sum_{i=0}%
^{\infty}b_{i,\ell}F^{i+1}T^{\ell}%
\]
and conclude that $b_{i,\ell}=0$ for all $i\geq d$. So%
\[
\alpha=\left(  \sum_{\ell}\sum_{i=0}^{d-1}b_{i,\ell}F^{i}T^{\ell}\right)  F\in
I^{\leq d-1}F\text{.}%
\]
This finishes the proof of the sub-claim.) Line \ref{lttt2} reads%
\[
I^{\leq d}\subseteq I^{\leq d_{0}}+IF\text{.}%
\]
\textit{(Inductive reduction process)} For $\alpha\in I^{\leq d}$ (with
$d>d_{0}$), we find some $\gamma\in I^{\leq d_{0}}$ such that%
\[
\alpha=\underset{\in IF}{(\alpha-\gamma)}+\underset{\in I^{\leq d_{0}}}%
{\gamma}\text{.}%
\]
Moving $\gamma$ to the other side of the equation, the left hand side lies in
$I^{\leq d}$, while the right-hand side lies in $IF$. Thus, both lie in the
intersection of both ideals. We deduce $\alpha-\gamma\in IF\cap I^{\leq
d}=I^{\leq d-1}F$ (by Equation \ref{lttt3}). In other words, we can promote
the above equation to%
\begin{equation}
\alpha=\underset{\in I^{\leq d-1}F}{(\alpha-\gamma)}+\underset{\in I^{\leq
d_{0}}}{\gamma}=\underset{\in I^{\leq d-1}}{\hat{\alpha}}\cdot F+\underset{\in
I^{\leq d_{0}}}{\gamma}\label{lm1}%
\end{equation}
for a suitably chosen $\hat{\alpha}\in I^{\leq d-1}$. Now repeat this
reduction process for the element $\hat{\alpha}$.\newline This is possible
unless we reach $d-1=d_{0}$. In this case, the above equation literally says
$\alpha\in I^{\leq d_{0}}\cdot R[F]$. Now, it follows that in Equation
\ref{lm1} we always get $\alpha\in I^{\leq d_{0}}\cdot R[F]$, for every
$\alpha$ in this inductive reduction process.\newline In particular, we
conclude $\alpha\in I^{\leq d_{0}}\cdot R[F]$ for our original $\alpha$ that
started the reduction process. As this works for all elements $\alpha\in
I^{\leq d}$, we obtain $I^{\leq d}\subseteq I^{\leq d_{0}}\cdot R[F]$. Taking
the union over all $d$, we obtain $I\subseteq I^{\leq d_{0}}\cdot R[F]$. But
we also have $I^{\leq d_{0}}\cdot R[F]\subseteq I$, so $I=I^{\leq d_{0}}\cdot
R[F]$ and since $I^{\leq d_{0}}$ is a finitely generated right $R[T]$-module
(by Lemma \ref{lem_fingen3}), it follows that $I^{\leq d_{0}}=\left\langle
a_{1},\ldots,a_{r}\right\rangle R[T]$ and thus $I=\left\langle a_{1}%
,\ldots,a_{r}\right\rangle R[T]\cdot R[F]$. Since $R[T]\cdot R[F]\subseteq
R[F][T]$, it follows that $I$ is a finitely generated $R[F][T]$-module.
\end{proof}

\begin{lemma}
\label{lem_finpres}Suppose

\begin{itemize}
\item $R$ is a Noetherian commutative $\mathbf{F}_{q}$-algebra,

\item $I$ is a right ideal in $R[F][T]$,

\item $R$ is $F$-finite.
\end{itemize}

A finitely generated right $R[F][T]$-module $M$ is finitely presented iff the
right $R[T]$-module $\ker(\widetilde{M}\overset{\cdot F}{\longrightarrow}M)$
is finitely generated.
\end{lemma}

\begin{proof}
As $M$ is a finitely generated right $R[F][T]$-module, we get a commutative
diagram%
\begin{equation}%
\begin{array}
[c]{rccccclc}%
0\rightarrow & \widetilde{K} & \rightarrow & \widetilde{R[F][T]}^{n} &
\rightarrow & \widetilde{M} & \rightarrow0 & \\
& \downarrow &  & \downarrow &  & \downarrow &  & \cdot F\\
0\rightarrow & K & \rightarrow & R[F][T]^{n} & \rightarrow & M &
\rightarrow0 &
\end{array}
\label{lah3}%
\end{equation}
and the $K$ on the left is just defined as the kernel. Therefore, by the snake
lemma%
\[
0\rightarrow\ker(\widetilde{M}\overset{\cdot F}{\rightarrow}M)\rightarrow
\operatorname*{coker}(\widetilde{K}\overset{\cdot F}{\rightarrow}K)\rightarrow
R[T]^{n}\rightarrow\operatorname*{coker}(\widetilde{M}\overset{\cdot
F}{\rightarrow}M)\rightarrow0
\]
is an exact sequence of right $R[F][T]$-modules (for the middle cokernel we
have used the isomorphism of line \ref{lah2}, the right action of $F$ on
$R[T]$ is tacitly understood to be the zero map). In other words, we get an
exact sequence of right $R[T]$-modules%
\[
0\rightarrow\ker(\widetilde{M}\overset{\cdot F}{\rightarrow}M)\rightarrow
\operatorname*{coker}(\widetilde{K}\overset{\cdot F}{\rightarrow}%
K)\rightarrow\ker\left(  R[T]^{n}\rightarrow\operatorname*{coker}%
(\widetilde{M}\overset{\cdot F}{\rightarrow}M)\right)  \rightarrow0\text{.}%
\]
As $R[T]^{n}$ is finitely generated and ($R$ and therefore) $R[T]$ is
Noetherian, it follows that the kernel on the right, is also right finitely
generated over $R[T]$. It follows that $\ker(\widetilde{M}\overset{\cdot
F}{\rightarrow}M)$ is finitely generated over $R[T]$ if and only if
$\operatorname*{coker}(\widetilde{K}\overset{\cdot F}{\rightarrow}K)$ is
finitely generated over $R[T]$. However, $K$ is a right $R[F][T]$-submodule of
$R[F][T]^{n}$, so by Lemma \ref{lem_rightemertonkeylemma} it is a finitely
generated right $R[F][T]$-module iff $\operatorname*{coker}(\widetilde
{K}\overset{\cdot F}{\rightarrow}K)$ is a finitely generated right
$R[T]$-module (the Lemma is about right ideals, but this is easily seen to be
equivalent to hold for all right submodules of free right modules). However,
by Diagram \ref{lah3}, $K$ being right finitely generated over $R[F][T]$, is
equivalent to $M$ being finitely presented as a right $R[F][T]$-module.
\end{proof}

The following is the analogue of Emerton's theorem, Theorem
\ref{marker_EmertonLeftCoherence}, just right instead of left, and stabilized
along additional polynomial variables.

\begin{proposition}
\label{prop_RightSuperCoherence}Suppose

\begin{itemize}
\item $R$ is a Noetherian commutative $\mathbf{F}_{q}$-algebra, and

\item $R$ is $F$-finite.
\end{itemize}

Then the ring $R[F]$ is right stably coherent.
\end{proposition}

\begin{proof}
\textit{(Step 1)} Firstly, we just prove that $R[F][T]$ is right coherent. Let
$I$ be a finitely generated right ideal in $R[F][T]$. Then by the snake lemma%
\[
I\hookrightarrow R[F][T]
\]
induces $\ker(\widetilde{I}\overset{\cdot F}{\rightarrow}I)\hookrightarrow
\ker(\widetilde{R[F][T]}\overset{\cdot F}{\rightarrow}R[F][T])=0$, so the
kernel is zero and thus certainly right finitely generated over $R[F][T]$. By
Lemma \ref{lem_finpres} it follows that $I$ is finitely presented. Thus, every
finitely generated right ideal is even finitely presented, i.e. $R[F][T]$ is
right coherent.\newline\textit{(Step 2)} We prove right stable coherence by
reducing to finitely many variables (this is inspired by the proof of
\cite[Thm. 1.8]{MR0396671}, who does the same for non-commuting variables):
Following Remark \ref{rmk_MultivariableRing}, Step 1 actually shows that
$R[F][T_{i}]_{i\in I}$ is right coherent for any \textit{finite} set $I$. Now,
for an arbitrary set $I$ we have the filtering colimit presentation%
\[
R[F][T_{i}]_{i\in I}=\underset{I^{\prime}\subseteq I\text{, }I^{\prime}\text{
finite}}{\underrightarrow{\operatorname*{colim}}}R[F][T_{i}]_{i\in I^{\prime}%
}\text{.}%
\]
For an arbitrary ring $A$, every non-zero free right $A$-module is faithfully
flat over $A$. Since the polynomial rings $A[T_{i}]_{i\in I}$ are right flat
over $A$ (unlike, in general, their Frobenius counterpart $A[F]$), it follows
that the above colimit has all transition morphisms (faithfully) flat. By
\cite[\S 5, Prop. 20]{MR0260799} a flat colimit of right coherent rings is
again right coherent. This shows that $R[F][T_{i}]_{i\in I}$ is right
coherent, regardless the choice of $I$.
\end{proof}

\subsection{Right regularity}

The following result is due to Linquan Ma \cite[Theorem 3.2]{MR3160413}:

\begin{proposition}
[Ma]\label{prop_GlobalRightDimOfFrobeniusSkewRing}Suppose that $R$ is a
commutative $\mathbf{F}_{q}$-algebra of finite global dimension. Then%
\[
\operatorname*{gldim}\nolimits_{\operatorname*{right}}R[F][T_{1},\ldots
,T_{s}]\leq\operatorname*{gldim}\nolimits_{\operatorname*{right}}R+s+1\text{.}%
\]

\end{proposition}

Ma's paper \cite{MR3160413} additionally assumes that $R$ is $F$-finite and
obtains an equality. However, for our purposes the stated estimate suffices.
As we want to cite an ingredient of the proof in the $K$-theory computations
later on, we give his proof adapted to our own notation:

Let $M$ be a right $R$-module. Then we can define a right $R[F]$-module%
\begin{equation}
M[X]:=M\otimes_{R}R[F]\text{.}\label{laaa1}%
\end{equation}
This can be spelled out as follows:
\[
M[X]=\left\{  \left.  \sum m_{i}X^{i}\right\vert m_{i}\in M\text{, }%
m_{i}=0\text{ for all but finitely many }i\right\}
\]%
\[
(m_{i}X^{i})\cdot rF^{j}:=m_{i}r^{q^{i}}X^{i+j}=(m_{i}\underset{M}{\cdot
}r^{q^{i}})X^{i+j}\text{.}%
\]
Now, suppose $N$ is an arbitrary right $R[F]$-module. Then, we can also
interpret it as a right $R$-module so that $N[X]$ makes sense. We define
$\phi:N[X]\rightarrow N$, $nX^{i}\mapsto nF^{i}$. This is a right
$R[F]$-module homomorphism. Since $nX^{0}\mapsto n$, for all $n\in N$, it is
surjective. Next, for every right $R[F]$-module $N$, we define a new right
$R[F]$-module $\tilde{N}$. Its additive group is just the one of $N$, but we
change the right action to $(\tilde{n})\underset{\tilde{N}}{\cdot}%
rF^{j}:=\tilde{n}\underset{N}{\cdot}r^{q}F^{j}$, where we refer to the
original right $R[F]$-module structure of $N$ on the right hand side. One
checks that this is indeed a right $R[F]$-module structure. Next, we define a
map $\psi:\tilde{N}[X]\rightarrow N[X]$, $nX^{i}\mapsto nX^{i+1}-nFX^{i}$. We
claim that this is a morphism of right $R[F]$-modules. This is a little
delicate because of the different module structures. We compute%
\[
(nX^{i})\underset{\tilde{N}[X]}{\cdot}rF^{j}=(n\underset{\tilde{N}}{\cdot
}r^{q^{i}})X^{i+j}=(n\underset{N}{\cdot}r^{q^{i+1}})X^{i+j}\text{,}%
\]
where on the right-hand side we just have the ordinary right $R[F]$-module
structure, so we could also just write $nr^{q^{i+1}}X^{i+j}$. Now,
$\psi(nr^{q^{i+1}}X^{i+j})=nr^{q^{i+1}}X^{i+j+1}-nr^{q^{i+1}}FX^{i+j}$. On the
other hand, $\psi(nX^{i})\underset{N[X]}{\cdot}rF^{j}=(nX^{i+1}-nFX^{i})\cdot
rF^{j}=nr^{q^{i+1}}X^{i+j+1}-nr^{q^{i+1}}FX^{i+j}$. Thus, it is indeed a right
$R[F]$-module homomorphism. This yields a `Frobenius-twisted Koszul complex':

\begin{lemma}
\label{lemma_FrobeniusTwistedKoszulComplex}For every right $R[F]$-module $N$,%
\begin{equation}
0\longrightarrow\tilde{N}[X]\overset{\psi}{\longrightarrow}N[X]\overset{\phi
}{\longrightarrow}N\longrightarrow0\label{laa1}%
\end{equation}
is a short exact sequence of right $R[F]$-modules.
\end{lemma}

\begin{proof}
We have $\phi\psi(\tilde{n}X^{i})=\phi(nX^{i+1}-nFX^{i})=nF^{i+1}-nF^{i+1}=0$.
We need to study the kernel of $\psi$. Suppose $\alpha=\sum_{i=0}^{m}%
n_{i}X^{i}$ lies in the kernel.%
\[
\psi(\sum_{i=0}^{m}n_{i}X^{i})=\sum_{i=0}^{m}n_{i}X^{i+1}-\sum_{i=0}^{m}%
n_{i}FX^{i}%
\]
and since the largest $X$-degree on the right-hand side, namely $X^{m+1}$, has
coefficient $n_{m}$, we must have $n_{m}=0$ if this element indeed lies in the
kernel of $\psi$. Thus, we could have started this computation with $m-1$
instead of $m$. By induction, it follows that $\alpha=0$. Finally, we need to
show exactness in the middle, i.e. that every element in the kernel of $\phi$
lies in the image of $\psi$. Suppose $\alpha=\sum_{i=0}^{m}n_{i}X^{i}$ lies in
$\ker\phi$ and $m\geq2$. Then%
\[
\alpha-\psi(n_{m}X^{m-1})=\sum_{i=0}^{m}n_{i}X^{i}-(n_{m}X^{m}-n_{m}FX^{m-1})
\]
is has $X$-degree strictly less than $\alpha$, and $\phi$ sends this new
element still to zero (since by assumption it sends $\alpha$ to zero, and we
already know that $\phi\psi$ is the zero map). Thus, by induction, we can
split off pre-images from $\alpha$ unless $m=1$. So let us restrict to this
case and say $\alpha=n_{1}X+n_{0}$. We find $\phi(\alpha)=n_{1}F+n_{0}$, so if
this is zero, we have $n_{0}=-n_{1}F$, i.e. $\alpha=n_{1}X-n_{1}F$.\ Now, we
obviously have $\psi(n_{1})=\alpha$, so this $\alpha$ also has a pre-image.
\end{proof}

\begin{proof}
[Proof of Prop. \ref{prop_GlobalRightDimOfFrobeniusSkewRing}]Suppose $s=0$. By
Equation \ref{laa1} there is a quasi-isomorphism from the $2$-term complex
$\left[  \tilde{N}[X]\rightarrow N[X]\right]  _{-1,0}$ to $N$. Thus, our claim
is proven once we can show that each right $R[F]$-module of the shape $N[X]$
has a projective resolution of length at most $n$. By taking the cone, this
then yields a projective resolution of $N$ of length at most $n+1$. Suppose
$N$ is a finitely generated right $R$-module. Note that if $0\rightarrow
A\rightarrow B\rightarrow C\rightarrow0$ is a short exact sequence, so is
$0\rightarrow A[X]\rightarrow B[X]\rightarrow C[X]\rightarrow0$ (as $R[F]$ is
flat as a left $R$-module, cf. Equation \ref{laaa1}). Moreover, there is an
isomorphism of right $R[F]$-modules $R[X]\rightarrow R[F]$, $rX^{i}\mapsto
rF^{i}$, showing that $R[X]$ is a rank one free $R[F]$-module, and thus if $R$
is a free $R$-module, $R[X]$ is a free $R$-module. If $N$ is a projective
$R$-module, $N\oplus M$ is a free $R$-module for a suitable $M$. Thus,
$N[X]\oplus M[X]$ is free. It follows that $N[X]$ is a projective right
$R[F]$-module. As a result of preserving exact sequences, this means that
every projective resolution of length $n$ of a right $R$-module $N$ induces a
projective resolution of length $n$ of the right $R[F]$-module $N[X]$.
Finally, we use the Hilbert syzygy theorem (\cite[Ch. 2, \S 5, (5.36),
Theorem]{MR1653294}) to deduce%
\[
\operatorname*{gldim}\nolimits_{\operatorname*{right}}R[F][T_{1},\ldots
,T_{s}]=\operatorname*{gldim}\nolimits_{\operatorname*{right}}R[F]+s\text{,}%
\]
proving the claim.
\end{proof}

\begin{lemma}
\label{lem_FinGenOverRMeansFinPresOverRF}Suppose $R$ is a commutative
$F$-finite Noetherian $\mathbf{F}_{q}$-algebra. If $N$ is a right
$R[F]$-module, which is finitely generated as a right $R$-module, then $N[X]$,
$\widetilde{N}[X]$ and $N$ are finitely presented right $R[F]$-modules.
\end{lemma}

\begin{proof}
As $N$ is finitely generated and $R$ Noetherian, $N$ is even finitely
presented, say $0\rightarrow K\rightarrow R^{\oplus m}\rightarrow
M\rightarrow0$ is exact with $K$ finitely generated. Hence, $0\rightarrow
K[X]\rightarrow R[X]^{\oplus m}\rightarrow N[X]\rightarrow0$ is also exact.
Since $R[X]\cong R[F]$ and because a finite set of generators of $K$ as an $R
$-module will induce a finite set of generators of $K[X]$ as a right
$R[F]$-module, it follows that $N[X]$ is a quotient of a free finite rank
$R[F]$-module by a finitely generated one, and thus $N[X]$ is finitely
presented. By Lemma \ref{lem_fingenA} (or Corollary \ref{cor_fingenA}) it
follows that $\widetilde{N}[X]$ is also finitely presented. Sequence
\ref{laa1} now implies that $N$ is the quotient of two finitely presented
modules, and thus itself finitely presented.
\end{proof}

\section{Sheaf-theoretic properties of the Frobenius skew ring}

\subsection{Construction}

It is very important that the Frobenius skew ring $R[F]$ can be turned into a
Zariski sheaf. This is well-known, but since it plays an important r\^{o}le
for us, let us recall some details. Again, Yoshino's paper \cite{MR1271618} is
an excellent reference. Firstly, if $S$ is a multiplicative set in $R$, it is
\textit{not} central in $R[F]$. Thus, it is a priori not even clear whether a
localization at $S$ exists. One needs to check the Ore conditions \cite[Ch. 9,
\S 9.1]{MR1098018}, \cite[\S 1.3]{MR1349108}, \cite[\S 10A]{MR1653294}.

\begin{lemma}
\label{lemma_multSetSInRIsBiOreInRF}Every multiplicative subset $S\subseteq R
$ satisfies the left- and right denominator set axioms as a multiplicative
subset of $R[F]$.
\end{lemma}

\begin{proof}
This is well-known. As it is absolutely crucial for the following, we prove
the two key axioms:\textit{\ (}$S$\textit{\ is left permutable)} For every
$s\in S$ and $r\in R$ there exists some $\tilde{s}\in S$ and $\tilde{r}\in R$
such that $\tilde{r}s=\tilde{s}r$. This can be checked as follows: Given
$r=\sum_{i=0}^{n}r_{i}F^{i}$ with $r_{i}\in R$ define $\rho_{i}:=r_{i}%
s^{q^{n}-q^{i}}$ and $\tilde{r}:=\sum_{i=0}^{n}\rho_{i}F^{i}$ as well as
$\tilde{s}:=s^{q^{n}}$. Then%
\[
\tilde{r}s=\sum\rho_{i}F^{i}s=\sum r_{i}s^{q^{n}-q^{i}}s^{q^{i}}F^{i}=\sum
r_{i}s^{q^{n}}F^{i}=s^{q^{n}}\sum r_{i}F^{i}=\tilde{s}r
\]
as desired.\footnote{Note that this proof only works because there exists some
finite degree $n$. It would fail for Frobenius skew power series $R[[F]]$.}
\textit{(}$S$\textit{\ is right permutable)} For every $s\in S$ and $r\in R$
there exists some $\tilde{s}\in S$ and $\tilde{r}\in R$ such that $r\tilde
{s}=s\tilde{r}$. Using the notation as before, define $\tilde{s}:=s$ and
$\rho_{i}:=r_{i}s^{q^{i}-1}$. For $i=0$, the latter means $\rho_{0}=r_{0}$.
Then%
\[
s\tilde{r}=\sum s\rho_{i}F^{i}=\sum r_{i}s^{q^{i}}F^{i}=\sum r_{i}%
F^{i}s=r\tilde{s}\text{.}%
\]
One also needs to check whether $S$ is left or right reversible, and we leave
the very similar verification to the reader (if $R$ is a domain, these
conditions are empty).
\end{proof}

These conditions being checked, it follows that the left localization
$S^{-1}R[F]$ and right localization $R[F]S^{-1}$ both exist. In fact, they agree:

\begin{lemma}
[{\cite[(4.10), Prop.]{MR1271618}}]%
\label{Lemma_DifferentLocalizationsForFrobSkewRing}Let $R$ be a commutative
ring. Suppose $S$ is a multiplicative subset. Then there are isomorphisms%
\[
R[F]S^{-1}\cong S^{-1}R[F]\cong(S^{-1}R)[F]\cong S^{-1}R\otimes_{R}R[F]\cong
R[F]\otimes_{R}S^{-1}R\text{.}%
\]

\end{lemma}

This lemma implies that $\mathcal{O}_{X}[F]$ can be turned into a
quasi-coherent sheaf of $\mathcal{O}_{X}$-modules; and for this it does not
matter whether we let $\mathcal{O}_{X}$ act from the left or the right (even
though these two actions are different!). It becomes a sheaf of $\mathcal{O}%
_{X}$-bimodules over $\mathbf{F}_{q}$, or equivalently a sheaf of left
$(\mathcal{O}_{X}\otimes_{\mathbf{F}_{q}}\mathcal{O}_{X})$-modules with%
\[
(\alpha,\beta)\cdot m:=\alpha m\beta\text{.}%
\]
Note that it is \textit{not} a sheaf of $\mathcal{O}_{X}$-algebras since the
natural inclusion $\mathcal{O}_{X}\hookrightarrow\mathcal{O}_{X}[F]$ as
constant polynomials does not lie in the center of the ring.

\begin{corollary}
$\mathcal{O}_{X}[F]$ is a quasi-coherent sheaf of $\mathcal{O}_{X}$-modules,
and this structure is indifferent to whether we let $\mathcal{O}_{X}$ act from
the left or from the right. It is also a sheaf of $\mathbf{F}_{q}$-algebras,
but (usually) not of $\mathcal{O}_{X}$-algebras.
\end{corollary}

\begin{remark}
In \cite{MR2071510} the sheaf $\mathcal{O}_{X}[F]$ is usually denoted by
$\mathcal{O}_{F,X}$.
\end{remark}

\subsection{Categories of sheaves}

Next, we shall introduce some categories mimicking ordinary coherent
$\mathcal{O}_{X}$-modules, but with an additional right action by the
Frobenius (so the action looks like a Cartier operator). Because the finite
presentation conditions will be with respect to $\mathcal{O}_{X}[F]$ instead
of $\mathcal{O}_{X}$, these can be much bigger than the coherent Cartier
modules of \cite{MR2863904}.

\begin{definition}
Let $X/\mathbf{F}_{q}$ be an $F$-finite Noetherian separated scheme.

\begin{enumerate}
\item Denote by $\mathsf{Coh}_{\mathcal{O}_{X}[F]}(X)$ the category whose
objects are locally finitely presented right $\mathcal{O}_{X}[F]$-module
sheaves, whose underlying (right) $\mathcal{O}_{X}$-module sheaves are
quasi-coherent. Morphisms are arbitrary right $\mathcal{O}_{X}[F]$-module morphisms.

\item If $Z\subseteq X$ is a closed subset, define the full subcategory
$\mathsf{Coh}_{\mathcal{O}_{X}[F],Z}(X)$ which consists of those right
$\mathcal{O}_{X}[F]$-module sheaves whose support, in terms of the underlying
right $\mathcal{O}_{X}$-module sheaf, is contained in $Z$.
\end{enumerate}
\end{definition}

As the rings $R[F]$ are right coherent, we will of course expect to get an
abelian category this way:

\begin{lemma}
\label{Lemma_OXFIsAbelianCategory}Let $X/\mathbf{F}_{q}$ be an $F$-finite
Noetherian separated scheme. Then

\begin{enumerate}
\item the category $\mathsf{Coh}_{\mathcal{O}_{X}[F]}(X)$ is an abelian
category, and

\item for every closed subset $Z\subseteq X$, the category $\mathsf{Coh}%
_{\mathcal{O}_{X}[F],Z}(X)$ a\ Serre subcategory. In particular,
$\mathsf{Coh}_{\mathcal{O}_{X}[F],Z}(X)$ is itself an abelian category.
\end{enumerate}
\end{lemma}

\begin{proof}
\textit{(1)} Affine locally, the ring $\mathcal{O}_{X}[F](U)=R[F]$ for
$R:=\mathcal{O}_{X}(U)$ is right coherent by Prop.
\ref{prop_RightSuperCoherence} and $X$ being $F$-finite. Affine locally, right
$\mathcal{O}_{X}[F]$-module sheaves then correspond to a finitely presented
right $R[F]$-modules. By Prop.
\ref{prop_CharacterizationsOfRightCoherentRings} the latter category is
abelian. Globally, the category of all right $\mathcal{O}_{X}[F]$-module
sheaves (with no more conditions) is Grothendieck abelian. One shows that
$\mathsf{Coh}_{\mathcal{O}_{X}[F]}(X)$ is closed under kernels and cokernels
in it. \textit{(2)} For any quasi-coherent sheaf of $\mathcal{O}_{X}$-modules
$\mathcal{F}$ the \emph{support} is defined as the set of scheme points $x\in
X$ such that $\mathcal{F}_{x}\neq0$ (this definition is more customary in the
context of coherent sheaves, where $\mathcal{F}$ is additionally known to be a
closed subset. For quasi-coherent sheaves the support can be arbitrary). It is
easy to see that for a short exact sequence $0\rightarrow\mathcal{F}^{\prime
}\rightarrow\mathcal{F}\rightarrow\mathcal{F}^{\prime\prime}\rightarrow0$ of
such sheaves, one has $\operatorname*{supp}\mathcal{F}=\operatorname*{supp}%
\mathcal{F}^{\prime}\cup\operatorname*{supp}\mathcal{F}^{\prime\prime}$. This
immediately implies that $\mathsf{Coh}_{\mathcal{O}_{X}[F],Z}(X)$ is a Serre subcategory.
\end{proof}

As in the classical case, the $K$-theory for these Frobenius skew-rings does
not see nil-thickenings:

\begin{lemma}
[{D\'{e}vissage for $\mathcal{O}_{X}[F]$}]\label{lemma_DevissageForOXF}Let
$X/\mathbf{F}_{q}$ be an $F$-finite Noetherian separated scheme,
$i:Z\hookrightarrow X$ a closed subscheme. Then the inclusion of categories%
\[
\mathsf{Coh}_{\mathcal{O}_{Z}[F]}(Z)\hookrightarrow\mathsf{Coh}_{\mathcal{O}%
_{X}[F],Z}(X)
\]
via pushforward, induces an equivalence in $K$-theory.%
\[
K(\mathsf{Coh}_{\mathcal{O}_{Z}[F]}(Z))\overset{\sim}{\longrightarrow
}K(\mathsf{Coh}_{\mathcal{O}_{X}[F],Z}(X))\text{.}%
\]

\end{lemma}

\begin{proof}
We adapt the proof of the analogous result for $\mathcal{O}_{X}$-module
sheaves. We begin with the following observation: Let $R$ be any commutative
$\mathbf{F}_{q}$-algebra and $I\subseteq R$ an ideal. Then the set%
\[
I\cdot R[F]=\left\{  \left.  \sum r_{i}F^{i}\right\vert r_{i}\in I\right\}
\]
is a two-sided ideal in $R[F]$ (since $s\cdot r_{i}F^{i}=(sr_{i})F^{i}$ and
$r_{i}F^{i}\cdot s=(r_{i}s^{q^{i}})F^{i}$ for all $s\in R$). We get%
\[
\iota^{\#}:R[F]\longrightarrow\left.  R[F]\right.  /\left.  I\cdot
R[F]\right.  \overset{\sim}{\longrightarrow}(R/I)[F]\text{,}%
\]
which generalizes to the respective sheaves of algebras $\mathcal{O}_{X}[F]$
and $\mathcal{O}_{Z}[F]$, i.e. we may read $\iota^{\#}$ as corresponding to
the closed immersion $i$ affine locally. Moreover, affine locally, the
pushforward interprets a right $(R/I)[F]$-module as a right $R[F]$-module via
$\iota^{\#}$. This defines an exact functor. Let $\mathcal{I}_{Z}$ denote the
ideal sheaf defining the closed immersion $i$. As the support for a sheaf in
$\mathsf{Coh}_{\mathcal{O}_{X}[F],Z}(X)$ was defined on the level of the
underlying quasi-coherent $\mathcal{O}_{X}$-module sheaf, every such sheaf can
be presented as an $\mathcal{O}_{X}[F]/\mathcal{I}_{Z}^{m}[F]\cong
(\mathcal{O}_{X}/\mathcal{I}_{Z}^{m})[F]$-module sheaf for some sufficiently
large $m\geq1$. In particular, we can filter the sheaf according to the powers
of $\mathcal{I}_{Z}^{m}[F]$. Thus, every object in $\mathsf{Coh}%
_{\mathcal{O}_{X}[F],Z}(X)$ admits a finite filtration whose graded pieces are
annihilated by $\mathcal{I}_{Z}^{m}[F]$, and thus lie in the full sub-category
of $\mathcal{I}_{Z}[F]$-annihilated sheaves. This category is non-empty, full,
closed under subobjects, quotients and finite direct sums. Thus, the
assumptions of Quillen's d\'{e}vissage theorem are satisfied, \cite[\S 5,
Theorem 4]{MR0338129}. However, the latter is also equivalent to the category
of finitely presented right $(\mathcal{O}_{X}/\mathcal{I}_{Z})[F]$-module
sheaves whose underlying $\mathcal{O}_{X}$-module sheaves are quasi-coherent,
i.e. $\mathsf{Coh}_{\mathcal{O}_{Z}[F]}(Z)$, giving the claim.
\end{proof}

\begin{remark}
It is important to work with $I\cdot R[F]$. Choosing the reverse order,
$R[F]\cdot I$, virtually never yields a two-sided ideal, only a left ideal.
\end{remark}

In order to proceed, it will be necessary to relate the quotient category%
\[
\mathsf{Coh}_{\mathcal{O}_{X}[F]}(X)/\mathsf{Coh}_{\mathcal{O}_{X}[F],Z}(X)
\]
to the open complement $U:=X-Z$. Recall the following fact due to Gabriel,
which explains the interplay of these categories in the classical case:

\begin{proposition}
[{\cite[\S III.5, Prop. VI.3]{MR0232821}}]\label{Prop_GabrielLocalization}%
Suppose $X$ is a scheme, $U\subseteq X$ an open subset such that the open
immersion $j:U\hookrightarrow X$ is a quasi-compact morphism.

\begin{enumerate}
\item Then for any quasi-coherent $\mathcal{O}_{U}$-module sheaf $\mathcal{F}
$ on $U$, the pushforward $j_{\ast}\mathcal{F}$ is a quasi-coherent
$\mathcal{O}_{X}$-module sheaf, and this functor is right adjoint to the
pullback $j^{\ast}$ so that%
\[
j^{\ast}\leftrightarrows j_{\ast}%
\]
is a pair of adjoint functors. Moreover, $j_{\ast}$ is fully faithful.

\item Let $\ker j^{\ast}$ be the full subcategory of objects $\mathcal{F}%
\in\mathsf{QCoh}(X)$ such that $j^{\ast}\mathcal{F}$ is a zero object. Then
$\ker j^{\ast}$ is precisely the Serre subcategory of quasi-coherent
$\mathcal{O}_{X}$-module sheaves with support contained in the closed set
$Z:=X-U$, i.e.%
\[
\ker j^{\ast}=\mathsf{QCoh}_{Z}(X)\text{.}%
\]

\item There is an equivalence of abelian categories%
\[
\mathsf{QCoh}(U)\overset{\sim}{\longrightarrow}\mathsf{QCoh}(X)/\mathsf{QCoh}%
_{Z}(X)\text{.}%
\]

\end{enumerate}
\end{proposition}

\begin{proposition}
[{Localization for $\mathcal{O}_{X}[F]$}]\label{prop_KFLocalizationSequence}%
Let $X/\mathbf{F}_{q}$ be an $F$-finite Noetherian separated scheme,
$i:Z\hookrightarrow X$ a reduced closed subscheme and $U:=X-Z$ the open
complement. Write $j:U\hookrightarrow X$ for the open immersion. Then there is
a homotopy fiber sequence%
\begin{equation}
K(\mathsf{Coh}_{\mathcal{O}_{Z}[F]}(Z))\overset{i_{\ast}}{\longrightarrow
}K(\mathsf{Coh}_{\mathcal{O}_{X}[F]}(X))\overset{j^{\ast}}{\longrightarrow
}K(\mathsf{Coh}_{\mathcal{O}_{U}[F]}(U))\longrightarrow+1\text{.}\label{lAM3}%
\end{equation}

\end{proposition}

\begin{proof}
This is of course just the Frobenius skew ring analogue of Quillen's
localization theorem \cite[Theorem 5]{MR0338129}. In Quillen's setup, he uses
that the coherent sheaves of $\mathcal{O}_{X}$-modules, $\mathsf{Coh}%
_{\mathcal{O}_{X},Z}(X)$, form a Serre subcategory of the abelian category
$\mathsf{Coh}_{\mathcal{O}_{X}}(X)$. By Gabriel's result, Prop.
\ref{Prop_GabrielLocalization}, the pullback along $j$ induces an equivalence
of abelian categories%
\[
j^{\ast}:\mathsf{Coh}_{\mathcal{O}_{X}}(X)/\mathsf{Coh}_{\mathcal{O}_{X}%
,Z}(X)\overset{\sim}{\longrightarrow}\mathsf{Coh}_{U}(U)\text{.}%
\]
The functor $j^{\ast}$ is exact since it is a localization and thus $(j^{\ast
}\mathcal{O}_{X})(V)$ is always a flat $\mathcal{O}_{X}(V)$-module for every
open $V\subseteq X$. Thus, in order to prove our claim above, we just need to
adapt Quillen's proof for the modules over the sheaf $\mathcal{O}_{X}[F]$
instead of $\mathcal{O}_{X}$. Let us follow this path step by step:\newline%
\textit{(Step 1)} By\ Lemma \ref{Lemma_OXFIsAbelianCategory} we again have a
Serre subcategory of an abelian category, so Quillen's localization theorem
\cite[Theorem 5]{MR0338129} for Serre subcategories in abelian categories
readily produces a homotopy fiber sequence%
\begin{equation}
K(\mathsf{Coh}_{\mathcal{O}_{X}[F],Z}(X))\overset{i_{\ast}}{\longrightarrow
}K(\mathsf{Coh}_{\mathcal{O}_{X}[F]}(X))\overset{q}{\longrightarrow
}K(\mathsf{Qu})\longrightarrow+1\text{,}\label{lAM2}%
\end{equation}
where $\mathsf{Qu}:=\mathsf{Coh}_{\mathcal{O}_{X}[F]}(X)/\mathsf{Coh}%
_{\mathcal{O}_{X}[F],Z}(X)$ denotes the quotient abelian category and $q$ the
canonical exact functor to the quotient category (see \cite[Ch. III, \S 1,
Prop. 1]{MR0232821} for the latter). In our situation, the functor%
\begin{equation}
j^{\ast}:\mathsf{Coh}_{\mathcal{O}_{X}[F]}(X)\longrightarrow\mathsf{Coh}%
_{\mathcal{O}_{U}[F]}(U)\label{lAM0}%
\end{equation}
is also exact because $(j^{\ast}\mathcal{O}_{X}[F])(V)$ is also always a flat
$\mathcal{O}_{X}[F](V)$-module thanks to Lemma
\ref{Lemma_DifferentLocalizationsForFrobSkewRing}. We claim that this functor
descends to an exact functor%
\begin{equation}
j^{\ast}:\mathsf{Qu}\longrightarrow\mathsf{Coh}_{\mathcal{O}_{U}%
[F]}(U)\text{.}\label{lAM1}%
\end{equation}
This is clear: We need to look at the underlying quasi-coherent $\mathcal{O}%
_{X}$-module sheaves. Quasi-coherent sheaves with support in $Z$ clearly go to
a zero object on $U$. Moreover, if a sheaf $\mathcal{F}\in\mathsf{Qu}$ lies in
the kernel of $j^{\ast}$, it also lies in the kernel of $j^{\ast}$ on the
level of the underlying quasi-coherent $\mathcal{O}_{X}$-module sheaves, so
again by Prop. \ref{Prop_GabrielLocalization}, $\mathcal{F}$ must have support
contained in $Z$. Hence, we obtain that $j^{\ast}$ in line \ref{lAM0} factors
over the quotient $\mathsf{Qu}$, giving line \ref{lAM1} (see \cite[Ch. III,
\S 1, Cor. 2]{MR0232821} for this factorization result). The functor $j^{\ast
}$ is also fully faithful since the latter reduces to the same property for
the underlying quasi-coherent $\mathcal{O}_{X}$-module sheaves.\newline%
\textit{(Step 2)}\ Next, we claim that $j^{\ast}:\mathsf{Qu}\rightarrow
\mathsf{Coh}_{\mathcal{O}_{U}[F]}(U)$ is essentially surjective. For this, one
can copy and adapt the standard proof that (on any quasi-separated and
quasi-compact scheme) every finitely presented $\mathcal{O}_{X}$-module sheaf
on $U$ has a finitely presented extension to $X$, being a pre-image under
$j^{\ast}$, \cite[Tag 01PD]{stacksproject}. Let $\mathcal{F}\in\mathsf{Coh}%
_{\mathcal{O}_{U}[F]}(U)$ be given. As our $X$ is even Noetherian, we can pick
a finite Zariski open cover $X=\bigcup_{i=1}^{\ell}V_{i}$. One can write
$X=U\cup\bigcup_{i=1}^{\ell}V_{i}$ and thus if one can successively extend
$\mathcal{F}$ from $U$ to $U\cup V_{1}$, and then to $U\cup V_{1}\cup V_{2}$
etc., our claim is proven. Thus, it suffices to deal with the case such that
$X=U\cup V$ with $V$ affine. For this problem, it suffices to extend
$\mathcal{F}$ from $U\cap V$ to $V$, since on $U\setminus(U\cap V)$ the sheaf
$\mathcal{F}$ is known and there are no glueing conditions that have to be
met. Moreover, we can reduce to the case where $U$ is affine, because we may
cover $U$ by such opens and if we solve the extension problem on these, we are
done. Now, as we have reduced the problem to an affine situation, we may cover
$U\cap V$ inside $V$ by a finite number of distinguished affine opens, say
$U\cap V=\bigcup_{j=1}^{k}D(f_{j})$, where $f_{j}$ are the functions whose
non-vanishing loci are the $D(f_{j})$. As $\mathcal{F}$ is known to be a
finitely generated right $\mathcal{O}_{U}[F]$-module by assumption, this
remains true on the opens $\iota_{j}:D(f_{j})\hookrightarrow U$. Now, by Lemma
\ref{Lemma_DifferentLocalizationsForFrobSkewRing} we have%
\begin{align*}
\mathcal{F}(D(f_{j}))  & =\mathcal{F}(U)\otimes_{\mathcal{O}_{U}%
[F]}\mathcal{O}_{D(f_{j})}[F]\\
& \cong\mathcal{F}(U)\otimes_{\mathcal{O}_{U}[F]}\mathcal{O}_{U}%
[F]\otimes_{\mathcal{O}_{U}}\mathcal{O}_{U}[\frac{1}{f_{j}}]\cong
\mathcal{F}(U)\otimes_{\mathcal{O}_{U}}\mathcal{O}_{U}[\frac{1}{f_{j}%
}]\text{.}%
\end{align*}
Thus, a finite set of generators will affine locally have the shape
$a_{i}\otimes\frac{1}{f_{j}^{r_{i}}}$ ($a_{i}\in\mathcal{F}(U)$, for a finite
index set $i\in I$), and so if we take the sub-module generated by the $a_{i}$
inside $\mathcal{F}(U)$, we get a finitely generated right $\mathcal{O}%
_{U}[F]$-module $\mathcal{G}_{j}\subseteq\mathcal{F}$ whose pullback to
$D(f_{j})$ satisfies $\iota_{j}^{\ast}\mathcal{G}_{j}\overset{\sim
}{\hookrightarrow}\iota_{j}^{\ast}\mathcal{F}$. In a similar way, one can lift
not just generators, but also a presentation, cf. \cite[Tag 01PD]%
{stacksproject}.\newline\textit{(Step 3)}\ Combining the previous steps, we
learn that the functor in line \ref{lAM1} is exact, fully faithful and
essentially surjective. Thus, it is an equivalence of categories. Thus, the
homotopy fiber sequence in line \ref{lAM2} transforms into%
\[
K(\mathsf{Coh}_{\mathcal{O}_{X}[F],Z}(X))\overset{i_{\ast}}{\longrightarrow
}K(\mathsf{Coh}_{\mathcal{O}_{X}[F]}(X))\overset{j^{\ast}}{\longrightarrow
}K(\mathsf{Coh}_{\mathcal{O}_{U}[F]}(U))\longrightarrow+1\text{.}%
\]
By our version of d\'{e}vissage, Lemma \ref{lemma_DevissageForOXF}, we arrive
at our claim in line \ref{lAM3}. This finishes the proof.
\end{proof}

The following is of course also entirely analogous to the corresponding
statement for the $K$-theory of coherent $\mathcal{O}_{X}$-module sheaves.

\begin{corollary}
\label{Cor_KFSheafOfSpectra_SatisfiesZariskiDescent}Let $X/\mathbf{F}_{q}$ be
an $F$-finite Noetherian separated scheme. The presheaf of spectra (see
\cite{MR2061851} for background),%
\[
KF(U):=K(\mathsf{Coh}_{\mathcal{O}_{U}[F]}(U))\text{,}%
\]
satisfies Zariski descent, and (equivalently) satisfies the Mayer--Vietoris
property for Zariski squares.
\end{corollary}

\begin{proof}
Satisfying\ Zariski descent is equivalent to the Mayer--Vietoris property
\cite[Ch. V, Theorem 10.2]{MR3076731} (the original results for this are due
to Brown and Gersten \cite{MR0347943}, based on Thomason's ideas). The latter
can be shown as in the case of ordinary $K$-theory and the presheaf of spectra
$U\mapsto K(\mathsf{Coh}_{\mathcal{O}_{X}}(U))$. For this, one just needs to
have a localization sequence, cf. \cite[Ch. V, Example 10.3]{MR3076731}: For
$X=U\cup V$ with $U,V\subseteq X$ open, define $Z:=X-U$ and one has the
closed-open complements $Z\hookrightarrow X\hookleftarrow U$ resp.
$Z\hookrightarrow V\hookleftarrow U\cap V$. Using the localization sequence of
Prop. \ref{prop_KFLocalizationSequence} for either, one obtains the two
homotopy fiber sequences%
\[%
\begin{array}
[c]{ccccccc}%
KF(Z) & \longrightarrow & KF(X) & \longrightarrow & K(U) & \longrightarrow &
+1\\
\parallel &  & \downarrow &  & \downarrow &  & \\
KF(Z) & \longrightarrow & KF(V) & \longrightarrow & KF(U\cap V) &
\longrightarrow & +1
\end{array}
\]
and the downward arrows are functorially induced from the pullback along the
open immersion $V\hookrightarrow X$. As the left downward arrow is an
equivalence, it follows that the square is homotopy bi-Cartesian. This is the
required Mayer--Vietoris property. This proof is an exact copy of the
classical one, just using the Frobenius variant of $K$-theory and the
respective localization sequence.
\end{proof}

\subsubsection{Variation: Frobenius vector
bundles\label{subsect_FrobeniusVectBundles}}

$K$-theory is usually defined on the basis of vector bundles or perfect
complexes. A similar treatment is also conceivable in the present situation.

\begin{lemma}
Let $X/\mathbf{F}_{q}$ be an $F$-finite Noetherian regular separated scheme.
Then%
\[
K(\left.  _{\mathcal{O}_{X}[F]}\mathsf{P}\right.  (X))\cong K(\left.
\mathsf{P}_{\mathcal{O}_{X}[F]}\right.  (X))\cong K(\mathsf{Coh}%
_{\mathcal{O}_{X}[F]}(X))\text{,}%
\]
where $\left.  \mathsf{P}_{\mathcal{O}_{X}[F]}\right.  $ (resp. $\left.
_{\mathcal{O}_{X}[F]}\mathsf{P}\right.  $) denotes the exact categories of
finitely generated projective right (resp. left) $\mathcal{O}_{X}[F]$-module sheaves.
\end{lemma}

\begin{proof}
The right regularity of the ring $\mathcal{O}_{X}[F](X)$, implied by the
uniform upper bound on projective dimension of Prop.
\ref{prop_GlobalRightDimOfFrobeniusSkewRing}, and its right coherence, Prop.
\ref{prop_RightSuperCoherence}, implies the right-hand side equivalence. Next,
we recall that for any arbitrary ring $A$, every finitely generated (right)
$A$-module is reflexive, i.e. the duality functor%
\[
(-)^{\vee}:\left.  \mathsf{P}_{\mathcal{O}_{X}[F]}\right.  (X)\rightarrow
\left.  _{\mathcal{O}_{X}[F]}\mathsf{P}\right.  (X)\text{,}\qquad M\mapsto
M^{\vee}:=\operatorname*{Hom}\nolimits_{A}(M,A)
\]
is an equivalence because $\left.  _{\mathcal{O}_{X}[F]}\mathsf{P}\right.
(X)$ has a corresponding duality functor, and double dualization either way
tautologically is an equivalence, thanks to the reflexivity of all objects in
the category. This induces an anti-equivalence between the left and right
module categories, and this induces an equivalence of their $K$-theories.
\end{proof}

\begin{remark}
I do not see any reason why one should expect that the $K$-theory of coherent
\emph{left} $\mathcal{O}_{X}[F]$-modules $\left.  _{\mathcal{O}_{X}%
[F]}\mathsf{Coh}(X)\right.  $ agrees with the above $K$-theories as well,
apart from aesthetics perhaps.
\end{remark}

It should also be possible to prove a localization sequence in this setup,
even for $F$-finite schemes which are not regular. The papers of Grayson
\cite{MR585217} and Weibel--Yao \cite{MR1156514}\ discuss a suitable
non-commutative localization theorem \textit{for rings}, i.e. in our context
this only helps us for the affine case.\ However, their constructions are
entirely functorial in the ring and the multiplicative subset. With a little
work, one can sheafify these localization results.

As for $\mathcal{O}_{X}$-modules, there is no counterpart of d\'{e}vissage if
one doesn't have a comparison to coherent sheaves available. We will not
pursue this further.

\subsection{Frobenius line invariance of $K$-theory}

A fundamental fact about algebraic $K$-theory is its $\mathbf{A}^{1}%
$-invariance, namely:

\begin{unnumthmrmk}
[Bass--Quillen]Suppose $X$ is a Noetherian regular separated scheme. Then the
pullback along the projection $\pi:X\times\mathbf{A}^{1}\rightarrow X$ induces
an equivalence%
\[
K(X)\overset{\sim}{\longrightarrow}K(X\times\mathbf{A}^{1})\text{.}%
\]
For the $K$-theory of coherent sheaves this remains true without the
assumption that $X$ be regular.
\end{unnumthmrmk}

This is \cite[Theorem 8 and Corollary]{MR0338129}. For $K_{0},K_{1}$ it is due
to Bass. Actually, this can be generalized to arbitrary affine fibrations
$\pi$, e.g. vector bundles. However, for us the above formulation is the
relevant one.

We shall prove:

\begin{theorem}
\label{thm_TwistedA1InvarianceOfKTheory}Let $X/\mathbf{F}_{q}$ be a Noetherian
separated $F$-finite regular scheme. Then there is an equivalence in
$K$-theory%
\[
\Psi:K(\mathsf{Coh}(X))\overset{\sim}{\longrightarrow}K(\mathsf{Coh}%
_{\mathcal{O}_{X}[F]}(X))\text{,}%
\]
induced by right tensoring $\mathcal{F}\mapsto\mathcal{F}\otimes
_{\mathcal{O}_{X}}\mathcal{O}_{X}[F]$.
\end{theorem}

\begin{philosophy}
Recall that over the base field $\mathbf{F}_{p}$ of characteristic $p$, the
Frobenius skew ring $\mathcal{O}_{X}[F]$ is defined by the relation
$r^{p}F=Fr$ for all sections $r$. If we were so bold to view a `field with one
element' as part of this family of fields, this relation would become $rF=Fr$,
i.e. $F$ would just be an ordinary commuting variable and thus $\mathcal{O}%
_{X}[T]$ just the ordinary polynomial ring. The above Frobenius line
invariance becomes ordinary $\mathbf{A}^{1}$-invariance. This aspect is too
magical to remain unsaid \cite{MR2074990}.\newline For perfect reduced rings,
where the Frobenius is an automorphism, Grayson uses a Frobenius-twisted
projective line in his article \cite{MR929766}, which serves an analogous purpose.
\end{philosophy}

To prove this, we use a particularly general version of Quillen's Theorem,
Theorem \ref{thm_Quillen}, which is due to Gersten \cite{MR0396671}. Gersten
managed to remove the Noetherian hypothesis and replace it by a condition on
the coherence of the polynomial ring:

\begin{unnumthmrmk}
[{\cite[\S 3, Theorem 3.1]{MR0396671}}]\label{thm_GerstenVersionOfQuillenThm}%
Let $A$ be an increasingly filtered ring%
\[
A=\bigcup_{s\geq0}A^{\leq s}%
\]
such that $A^{\leq s}\cdot A^{\leq t}\subseteq A^{\leq s+t}$. Suppose that the
polynomial ring $A[T]$ is right coherent and right regular. Then the ring
homomorphism $A^{\leq0}\hookrightarrow A$ induces an equivalence in $K$-theory%
\[
K(\mathsf{P}_{f}(A^{\leq0}))\overset{\sim}{\longrightarrow}K(\mathsf{P}%
_{f}(A))\text{.}%
\]

\end{unnumthmrmk}

Moreover:

\begin{unnumthmrmk}
[{\cite[\S 2, Theorem 2.3]{MR0396671}}]%
\label{thm_GerstenCompareProjectivesWithFinPres}Suppose $A$ is a right
coherent and right regular ring. Then the exact functor%
\[
\mathsf{P}_{f}(A)\longrightarrow\mathsf{Mod}_{fp}(A)
\]
induces an equivalence in $K$-theory.
\end{unnumthmrmk}

To make sense of the functor, note that the right coherence of $A$ implies
that every finitely generated projective right module is actually coherent
(Prop. \ref{prop_PropsAroundCoherence}). We refer to Gersten's paper for the proofs.

\begin{proof}
[Proof of Theorem \ref{thm_TwistedA1InvarianceOfKTheory}]\textit{(Step 1)} We
first deal with the special case that $X=\operatorname*{Spec}R $ is affine.
Then we have equivalences of abelian categories%
\begin{equation}
K(\mathsf{Coh}(X))\overset{\sim}{\longrightarrow}K(\mathsf{Mod}_{fp}%
(R))\text{,}\qquad K(\mathsf{Coh}_{\mathcal{O}_{X}[F]}(X))\overset{\sim
}{\longrightarrow}K(\mathsf{Mod}_{fp}(R[F]))\text{.}\label{lvia1}%
\end{equation}
On the left-hand side, it does not matter whether we deal with finitely
generated or finitely presented modules since $R$ is Noetherian by assumption.
Clearly the skew Frobenius ring is filtered%
\[
R[F]^{\leq d}=\left\{  \left.  \sum\nolimits_{i=0}^{d}r_{i}F^{i}\right\vert
r_{i}\in R\right\}  \qquad\text{with}\qquad R[F]^{\leq0}\cong R\text{.}%
\]
By Prop. \ref{prop_RightSuperCoherence} the polynomial ring $R[F][T]$ is right
coherent (this is the crucial bit why just proving right coherence for $R[F]$
would not have sufficed!) and by Prop.
\ref{prop_GlobalRightDimOfFrobeniusSkewRing} the regularity of $R$ (which
since $R$ is Noetherian commutative by Serre implies that its global dimension
agrees with its Krull dimension and is finite, \cite[(5.94)\ Theorem]%
{MR1653294}) implies the finiteness of the right global dimension of $R[F][T]
$ and thus implies right regularity. Thus, the conditions of Gersten's theorem
(Theorem \ref{thm_GerstenVersionOfQuillenThm}) are met, i.e. we get an
equivalence in $K$-theory%
\[
K(\mathsf{P}_{f}(R))\overset{\sim}{\longrightarrow}K(\mathsf{P}_{f}(R[F]))
\]
and by Theorem \ref{thm_GerstenCompareProjectivesWithFinPres} this can be
transformed into an equivalence%
\[
K(\mathsf{Mod}_{fp}(R))\overset{\sim}{\longrightarrow}K(\mathsf{Mod}%
_{fp}(R[F]))\text{.}%
\]
Combining this with the equivalences in line \ref{lvia1} gives our claim in
the affine case.\newline\textit{(Step 2)} Suppose $X$ is not affine. Two proof
ideas come to mind: \textit{(Version I)} The presheaf of spectra $K:U\mapsto
K(\mathsf{Coh}(U))$ satisfies Zariski descent. By Corollary
\ref{Cor_KFSheafOfSpectra_SatisfiesZariskiDescent} the same is true for its
Frobenius line analogue $KF:U\mapsto K(\mathsf{Coh}_{\mathcal{O}_{U}[F]}(U))$.
Hence, by\ Zariski descent we may check whether the induced homomorphism%
\[
K\rightarrow KF\text{,}\qquad\mathcal{F}\mapsto\mathcal{F}\otimes
_{\mathcal{O}_{X}}\mathcal{O}_{X}[F]
\]
is an equivalence on affine covers, where it is true by Step 1.
\textit{(Version II)} We can also circumvent sheaf methods and follow
Quillen's lead in \cite[\S 7.4, Prop. 4.1]{MR0338129}. We prove the claim by
induction on the dimension of $X$. If $X$ is zero-dimensional, it is just a
collection of finitely many closed points and we are in the affine situation.
This case is already proven. Thus, suppose the case of dimension $n-1$ is
settled and $\dim X=n\geq1$. Suppose $Z\hookrightarrow X$ is a reduced closed
subscheme with $\operatorname*{codim}_{X}Z\geq1$ and $U:=X-Z$ its open
complement. Then we have the localization sequence, Proposition
\ref{prop_KFLocalizationSequence},%
\[
K(\mathsf{Coh}_{\mathcal{O}_{Z}[F]}(Z))\overset{i_{\ast}}{\longrightarrow
}K(\mathsf{Coh}_{\mathcal{O}_{X}[F]}(X))\overset{j^{\ast}}{\longrightarrow
}K(\mathsf{Coh}_{\mathcal{O}_{U}[F]}(U))\longrightarrow+1\text{.}%
\]
As $K$-theory commutes with filtering colimits, we may run over the filtering
family of all reduced closed subschemes $Z\hookrightarrow X$ such that
$\operatorname*{codim}_{X}Z\geq1$, ordered by inclusion of the underlying
closed subsets in the Zariski topology. We obtain%
\[
\underset{Z}{\underrightarrow{\operatorname*{colim}}}K(\mathsf{Coh}%
_{\mathcal{O}_{Z}[F]})\rightarrow K(\mathsf{Coh}_{\mathcal{O}_{X}%
[F]})\rightarrow\coprod_{x\in X^{0}}K(\mathsf{Coh}_{\mathcal{O}_{X,x}%
[F]})\rightarrow+1\text{.}%
\]
Each entry in the first term agrees with $K(\mathsf{Coh}(Z))$ by our induction
hypothesis and the fact that $Z$ has dimension at most $n-1$. Moreover, as the
$x\in X^{0}$ are generic points, $\dim\mathcal{O}_{X,x}=0$ and thus
$K(\mathsf{Coh}_{\mathcal{O}_{X,x}[F]})=K(\mathsf{Coh}(\kappa(x)))$, because
we are in an affine situation and can use Step 1. Now conclude by the
Five\ Lemma as in Quillen's proof of \cite[\S 7.4, Prop. 4.1]{MR0338129}. By
induction, the proof is finished. Unravelling the proof of Zariski descent for
$KF$, it is easy to see that both versions of the proof just differ by the
geometric intuition employed, yet on the technical level they are entirely equivalent.
\end{proof}

\section{Cartier modules}

We recall the basic definitions of the theory of Cartier modules, due to
Blickle and B\"{o}ckle. See \cite[\S 2]{MR2863904} for details and proofs.
Suppose $X$ is a Noetherian scheme, separated over $\mathbf{F}_{q}$. We write
$F$ for the Frobenius morphism%
\[
F:X\rightarrow X\text{,}%
\]
which maps $f$ to $f^{q}$ on the level of the structure sheaves, and is the
identity map on the underlying topological space of $X$. A \emph{coherent}%
\ \emph{Cartier module} (or in alternative terminology: a `coherent $\kappa
$-sheaf', see \cite[\S 2]{BBCartierCrystals}) is a coherent $\mathcal{O}_{X}%
$-module sheaf $\mathcal{F}$ along with an $\mathcal{O}_{X}$-linear map
$C:F_{\ast}\mathcal{F}\rightarrow\mathcal{F}$. A morphism of Cartier modules
is a morphism $\psi:\mathcal{F}\rightarrow\mathcal{G}$ of the underlying
coherent sheaves commuting with the respective maps $C$, i.e. the diagram%
\[%
\begin{array}
[c]{ccc}%
F_{\ast}\mathcal{F} & \overset{F_{\ast}\psi}{\longrightarrow} & F_{\ast
}\mathcal{G}\\
\downarrow &  & \downarrow\\
\mathcal{F} & \underset{\psi}{\longrightarrow} & \mathcal{G}%
\end{array}
\]
is supposed to commute. We write $\mathsf{CohCart}(X)$ for the category of
coherent Cartier modules. This is an abelian category. The Cartier module
$\mathcal{F}$ is called \emph{nilpotent} if there exists some integer $v\geq1
$ such that the $v$-th power $C^{v}:F_{\ast}^{v}\mathcal{F}\rightarrow
\mathcal{F}$ is the zero morphism. One can show that $\mathcal{F}$ is
nilpotent if and only if this holds for the stalks $\mathcal{F}_{x}$ at all
closed points $x\in X$, \cite[Lemma 2.10]{MR2863904}.

The category of nilpotent coherent Cartier modules, call it $\mathsf{CohCart}%
_{nil}(X)$, is a Serre subcategory of $\mathsf{CohCart}(X)$, \cite[Lemma
2.11]{MR2863904}. Define the category of \emph{Cartier crystals} as the
quotient category%
\[
\mathsf{CartCrys}(X):=\mathsf{CohCart}(X)/\mathsf{CohCart}_{nil}(X)\text{.}%
\]
Since this is a quotient of an abelian category by a Serre category, this is
itself an abelian category. These categories have a very rich theory of
pullback and pushforward functors, both in $\ast$- and $!$-variants. This is a
longer story and since we shall not need them, we refer the curious reader to
\cite{MR2863904}.

\subsection{Riemann--Hilbert correspondence\label{subsect_RH}}

We shall need a positive characteristic version of the Riemann--Hilbert
correspondence: Suppose $X/\mathbf{F}_{q}$ is a smooth scheme. Let
$\mathsf{\acute{E}t}_{c}(X,\mathbf{F}_{q})$ be the abelian category of
constructible \'{e}tale sheaves with $\mathbf{F}_{q}$-coefficients. The
derived category $D_{c}^{b}(X_{\acute{e}t},\mathbf{F}_{q})$ has a perverse
$t$-structure, constructed by Gabber \cite{MR2099084}, \cite[\S 11.5.2]{MR2071510}. Let
$\mathsf{\acute{E}t}_{perv}(X,\mathbf{F}_{q})$ denote the heart of this
$t$-structure. The objects of this category are called \emph{perverse
\'{e}tale sheaves}.

Then there is an equivalence of abelian categories%
\begin{equation}
\operatorname*{Sol}:\mathsf{CartCrys}(X)\overset{\sim}{\longrightarrow
}\mathsf{\acute{E}t}_{perv}(X,\mathbf{F}_{q})^{op}\text{.}\label{lcinn1}%
\end{equation}
This result can be obtained by combining the Riemann--Hilbert correspondence
of Emerton--Kisin with a mechanism due to Blickle and B\"{o}ckle. We give a
summary of how this works:

\begin{definition}
Suppose $X/\mathbf{F}_{q}$ is regular and $F$-finite. Then we call it
$F$\emph{-dualizing}, if the dualizing sheaf $\omega_{X}$ (relative to the
base $\mathbf{F}_{q}$) is compatible with the Frobenius in the sense of
$F^{!}\omega_{X}\simeq\omega_{X}$.
\end{definition}

This condition is discussed in detail in \cite[\S 2.4]{MR2863904}. Moreover,
if $X/\mathbf{F}_{q}$ is smooth or $X$ is affine, $X$ is automatically
$F$-dualizing, so this condition is harmless. See loc. cit.

We write $\mu_{\operatorname*{lfgu}}(X)$ for the abelian category of locally
finitely generated unit left $\mathcal{O}_{X}[F]$-modules, as introduced by
Emerton and\ Kisin (\cite[Definition 6.3]{MR2071510}).

\begin{theorem}
[Emerton--Kisin]Suppose $X/\mathbf{F}_{q}$ is a smooth scheme. Then there is
an equivalence of abelian categories%
\[
\mu_{\operatorname*{lfgu}}(X)\overset{\sim}{\longrightarrow}\mathsf{\acute
{E}t}_{perv}(X,\mathbf{F}_{q})^{op}\text{.}%
\]

\end{theorem}

This is the central result of the first part of \cite{MR2071510}. This is also
explained in \cite{MR2099082}.

\begin{theorem}
[Blickle--B\"{o}ckle \cite{MR2863904}]%
\label{thm_BlickleBoeckleRHCartCrysToUnitFCrys}Suppose $X/\mathbf{F}_{q}$ is a
regular, $F$-finite and $F$-dualizing scheme. Then there is an equivalence of
abelian categories%
\[
\mathsf{CartCrys}(X)\overset{\sim}{\longrightarrow}\mu_{\operatorname*{lfgu}%
}(X)\text{.}%
\]

\end{theorem}

This equivalence is constructed by concatenating a substantial list of other
individual equivalences of categories. Let us sketch this:\textit{\ (Step 1)}
There is another category of crystals, $\gamma$-crystals, which we denote by
$\left.  \gamma\text{-}\mathsf{Crys}\right.  (X)$. It is a quotient category
of the $\gamma$-sheaf category of \cite{MR2407119}. We shall not use it
anywhere outside this section. Now, if $X/\mathbf{F}_{q}$ is regular and
$F$-finite and $F$-dualizing, then Blickle and B\"{o}ckle produce an
equivalence of abelian categories%
\begin{align*}
\mathbb{D}:\mathsf{CartCrys}(X)  & \longrightarrow\left.  \gamma
\text{-}\mathsf{Crys}\right.  (X)\\
\mathcal{F}  & \longmapsto\mathcal{F}\otimes_{\mathcal{O}_{X}}\omega_{X}%
^{\vee}\text{,}%
\end{align*}
where $\omega_{X}^{\vee}$ is the tensor inverse of $\omega_{X}$. See
\cite[Theorem 5.9]{MR2863904} for the proof. The functor clearly makes sense,
sending Cartier modules to $\gamma$-sheaves, and then they show that it
factors over the corresponding notions of nilpotent Cartier resp. nilpotent
$\gamma$-sheaves. Thus, it descends to a functor between the associated
quotient categories of crystals. \textit{(Step 2)} The second step is already
explained in Blickle's paper \cite{MR2407119}. Besides introducing $\gamma
$-sheaves at all, he establishes an equivalence of abelian categories%
\begin{align*}
\mathsf{Gen}:\left.  \gamma\text{-}\mathsf{Crys}\right.  (X)  &
\longrightarrow\mu_{\operatorname*{lfgu}}(X)\\
\mathcal{F}  & \longmapsto\operatorname*{colim}\left(  \mathcal{F}\rightarrow
F^{\ast}\mathcal{F}\rightarrow F^{2\ast}\mathcal{F}\rightarrow F^{3\ast
}\mathcal{F}\rightarrow\cdots\right)  \text{,}%
\end{align*}
where the transition morphisms are part of the datum of a $\gamma$-sheaf. By
\cite[Theorem 2.27]{MR2407119} such an equivalence exists with `minimal
$\gamma$-sheaves' on the left-hand side, and by \cite[Prop. 3.1]{MR2407119}
the latter category is equivalent to $\left.  \gamma\text{-}\mathsf{Crys}%
\right.  (X)$. The above functor is defined as the composition of these
equivalences. The essential surjectivity is the existence statement for
generators, dating back to the start of the entire business \cite{MR1476089}.

\begin{theorem}
\label{thm_TauCrystalsAgreeWithCartierCrystals}The category of $\tau$-crystals
of \cite{MR2561048}, which we will denote by $\left.  \tau\text{-}%
\mathsf{Crys}\right.  (X)$, is also equivalent to $\mathsf{CartCrys}(X)$.
\end{theorem}

This follows by combining the Riemann--Hilbert correspondence of B\"{o}ckle
and\ Pink \cite[Theorem 10.3.6]{MR2561048} with the above. As far as I
understand, a direct construction of this equivalence will appear in
\cite{bbforthcoming}.

There is a much bigger panorama involving even more categories and their
equivalence and duality relations. We refer to \cite{MR2863904},
\cite{BBCartierCrystals} for the big picture.

\begin{remark}
[Singular case]Recently, the Riemann--Hilbert correspondence in the
formulation of line \ref{lcinn1} has even been shown for singular $X$ in
Schedlmeier's thesis \cite{Schedl}, under rather mild hypotheses. It is proven
by reduction to the smooth case. Presumably, a part of the following results
could be extended to singular $X$ using this.
\end{remark}

\subsection{$K$-theory of Cartier crystals\label{subsect_KTheoryOfCartCrys}}

Next, we compute the $K$-theory of the categories of Cartier modules and
Cartier crystals.

\begin{proposition}
\label{prop_Coh_Equals_CohCartNilp}There is a canonical equivalence%
\[
K_{m}(\mathsf{Coh}(X))\overset{\sim}{\longrightarrow}K_{m}(\mathsf{CohCart}%
_{nil}(X))\text{.}%
\]
It is induced from the exact functor which sends a coherent sheaf
$\mathcal{F}$ to the Cartier module, where the Cartier operator $C$ acts as
the zero map.
\end{proposition}

\begin{proof}
It is clear that the relevant functor exists and is exact. Every nilpotent
Cartier module can be filtered and its filtered parts have trivial action by
$F$. Thus, each filtered part lies in the essential image of the functor. It
follows from Quillen's d\'{e}vissage theorem, \cite[\S 5, Theorem
4]{MR0338129}, that the induced map in $K$-theory is an equivalence.
\end{proof}

\begin{proposition}
\label{Prop_SplitingOfKTheoryOfCohCart}There is a canonical long exact
sequence%
\[
\cdots\rightarrow K_{m}(\mathsf{Coh}(X))\rightarrow K_{m}(\mathsf{CohCart}%
(X))\overset{q}{\rightarrow}K_{m}(\mathsf{CartCrys}(X))\rightarrow\cdots
\]
for all $m\in\mathbf{Z}$.
\end{proposition}

\begin{proof}
We have the exact sequence of abelian categories,%
\[
\mathsf{CohCart}_{nil}(X)\hookrightarrow\mathsf{CohCart}(X)\overset
{q}{\twoheadrightarrow}\mathsf{CartCrys}(X)\text{.}%
\]
Thus, by Quillen's localization sequence, \cite[\S 5, Theorem 5]{MR0338129},
there is a long exact sequence in $K$-theory groups,%
\begin{align}
& \cdots\rightarrow K_{m}(\mathsf{CohCart}_{nil}(X))\rightarrow K_{m}%
(\mathsf{CohCart}(X))\rightarrow\label{l1}\\
& \qquad\qquad\cdots\overset{q}{\rightarrow}K_{m}(\mathsf{CartCrys}%
(X))\rightarrow K_{m-1}(\mathsf{CohCart}_{nil}(X))\rightarrow\cdots
\text{.}\nonumber
\end{align}
The result follows from combining this with Prop.
\ref{prop_Coh_Equals_CohCartNilp}.
\end{proof}

\begin{definition}
Being $F$-finite, the Frobenius is a finite morphism and we get an exact
functor of pushforward%
\[
F_{\ast}:K(\mathsf{Coh}(X))\longrightarrow K(\mathsf{Coh}(X))\text{.}%
\]

\end{definition}

There is an exact forgetful functor%
\[
U:\mathsf{CohCart}(X)\longrightarrow\mathsf{Coh}(X)
\]
which just forgets the right action by $F$, i.e. it sends a coherent Cartier
module just to its underlying coherent $\mathcal{O}_{X}$-module. We write
\textquotedblleft$U$\textquotedblright\ for \textquotedblleft
underlying\textquotedblright.

\begin{proposition}
\label{prop_CompareEquivalenceAndCategoryInclusion}On the level of $K$-theory,
the inclusion functor%
\begin{equation}
\mathsf{CohCart}(X)\longrightarrow\mathsf{Coh}_{\mathcal{O}_{X}[F]}%
(X)\text{,}\label{lBM3}%
\end{equation}
when composed with the inverse $\Psi^{-1}$ (of the equivalence of Theorem
\ref{thm_TwistedA1InvarianceOfKTheory}), i.e.%
\[
\mathsf{CohCart}(X)\longrightarrow\mathsf{Coh}_{\mathcal{O}_{X}[F]}%
(X)\underset{\Psi}{\overset{\sim}{\longleftarrow}}\mathsf{Coh}(X)\text{,}%
\]
agrees with $(1-F_{\ast})\circ U$.
\end{proposition}

\begin{proof}
Given any coherent Cartier module $\mathcal{F}$, we have an exact sequence of
right $\mathcal{O}_{X}[F]$-modules,%
\[
0\longrightarrow\widetilde{\mathcal{F}}[X]\longrightarrow\mathcal{F}%
[X]\longrightarrow\mathcal{F}\longrightarrow0\text{,}%
\]
by using the sheaf version of Lemma \ref{lemma_FrobeniusTwistedKoszulComplex}.
By Lemma \ref{lem_FinGenOverRMeansFinPresOverRF} this is a short exact
sequence in the category $\mathsf{Coh}_{\mathcal{O}_{X}[F]}(X)$. We get three
functors $s_{i}:\mathsf{CohCart}(X)\rightarrow\mathsf{Coh}_{\mathcal{O}%
_{X}[F]}(X)$ for $i=1,2,3$:%
\[
\mathcal{F}\longrightarrow\widetilde{\mathcal{F}}[X]\text{,}\qquad
\mathcal{F}\longrightarrow\mathcal{F}[X]\text{,}\qquad\mathcal{F}%
\longrightarrow\mathcal{F}\text{.}%
\]
The last functor is obviously exact. The middle functor can also be written as
$\mathcal{F}[X]=\mathcal{F}\otimes_{\mathcal{O}_{X}}\mathcal{O}_{X}[F]$ and
since $\mathcal{O}_{X}[F]$ is a free and thus flat left $\mathcal{O}_{X}%
$-module sheaf, this functor is also exact. The leftmost functor arises
analogously, intervowen with the exact pushforward $F_{\ast}$, so it is also
an exact functor. By the Additivity Theorem \cite[\S 3]{MR0338129}, it follows
that $s_{2\ast}=s_{1\ast}+s_{3\ast}$ on $K$-theory. We note that $s_{2\ast
}=\Psi\circ U$ (by the very definition; recall the explicit functor of Theorem
\ref{thm_TwistedA1InvarianceOfKTheory}. Note that $\mathcal{F}\otimes
_{\mathcal{O}_{X}}\mathcal{O}_{X}[F]$ only depends on the (right)
$\mathcal{O}_{X}$-module structure of $\mathcal{F}$, so we tacitly have
forgotten its right-action by $F$) and by Remark \ref{remark_on_notation},
$s_{1\ast}=\Psi\circ F_{\ast}\circ U$. Thus, $s_{3\ast}=\Psi\circ(1-F_{\ast
})\circ U$, but $s_{3\ast}$ is just the inclusion of line \ref{lBM3}. This
proves the claim.
\end{proof}

Next, we recall the following classical result in algebra:

\begin{unnumthmrmk}
[Wedderburn's Theorem]Every finite division ring is a field.
\end{unnumthmrmk}

\begin{proposition}
\label{prop_K_of_CartCrys_InPrimeToPTorsion}Suppose $X$ is $F$-finite.

\begin{enumerate}
\item For $m\geq1$ the group $K_{m}(\mathsf{CartCrys}(X))$ is a pure torsion
group, of order prime to $p$. In particular,%
\[
K_{m}(\mathsf{CartCrys}(X))\otimes_{\mathbf{Z}}\mathbf{Z}_{(p)}=0\qquad
\text{for}\qquad m\geq1\text{,}%
\]

\item $K_{m}(\mathsf{CartCrys}(X))=0$ for $m=2i$, $i>1$, and

\item $K_{0}(\mathsf{CartCrys}(X))$ is the free abelian group whose generators
correspond to the iso-classes of simple Cartier crystals.
\end{enumerate}
\end{proposition}

\begin{proof}
Firstly, we use that if $X$ is $F$-finite, then the abelian category
$\mathsf{CartCrys}(X)$ is both Noetherian and Artinian by \cite[Corollary
4.7]{MR2863904}. In particular, every object in it has finite length. Let
$\mathsf{CartCrys}(X)^{ss}$ be the full subcategory of simple objects, i.e.
objects which do not admit any non-trivial subobjects. By\ Quillen's
d\'{e}vissage theorem, \cite[\S 5, Theorem 4]{MR0338129}, the inclusion
functor $i:\mathsf{CartCrys}(X)^{ss}\hookrightarrow\mathsf{CartCrys}(X)$
induces an equivalence of $K$-groups,%
\begin{equation}
K_{m}(\mathsf{CartCrys}(X)^{ss})\overset{\sim}{\longrightarrow}K_{m}%
(\mathsf{CartCrys}(X))\label{l3}%
\end{equation}
for all $m\in\mathbf{Z}$. Being semisimple abelian, we have an equivalence of
$\mathbf{F}_{q}$-linear abelian categories%
\begin{equation}
\mathsf{CartCrys}(X)^{ss}\longrightarrow\coprod_{Z}\mathsf{Mod}(D_{Z}%
)\text{,}\label{l4}%
\end{equation}
where $Z$ runs through the simple objects, and $D_{Z}:=\operatorname*{End}%
_{\mathsf{CartCrys}}(Z)$ denotes their endomorphism algebras. These must be
finite fields by \cite[Corollary 4.16]{MR2863904} $-$ in detail: Being simple,
Schur's Lemma implies that $D_{Z}$ is a division ring over $\mathbf{F}_{q}$.
Coherence of $Z$ and reduction to the single underlying associated prime
implies that $D_{Z}$ is also a finite-dimensional $\mathbf{F}_{q}$-vector
space. Thus, by\ Wedderburn's Theorem, the finite division ring $D_{Z}$ must
be a (necessarily finite)\ field. We obtain%
\[
D_{Z}\simeq\mathbf{F}_{p^{r(Z)}}\qquad\text{(for }r(Z)\geq1\text{
appropriately chosen depending on }Z\text{)}%
\]
and therefore line \ref{l4} implies that%
\[
K_{m}(\mathsf{CartCrys}(X)^{ss})\cong\coprod_{Z}K_{m}(D_{Z})=\coprod_{Z}%
K_{m}(\mathbf{F}_{p^{r(Z)}})
\]
and by Quillen's computation of the $K$-theory of finite fields
\cite{MR0315016},%
\[
K_{m}(\mathbf{F}_{p^{r}})=\left\{
\begin{array}
[c]{ll}%
\mathbf{Z}/(p^{ri}-1) & \text{for }m\geq1\text{ odd, }m=2i-1\\
0 & \text{for }m\geq1\text{ even.}%
\end{array}
\right.
\]
it follows that for $m\geq1$ the group $K_{m}(\mathsf{CartCrys}(X)^{ss})$ is a
direct sum of prime-to-$p$ pure torsion groups, and vanishes for even $m\geq
2$. Moreover, $K_{0}$ of any field is $\mathbf{Z}$. Along with the isomorphism
in line \ref{l3} this yields our claim.
\end{proof}

A different description of the $K_{0}$-group has been developed by Taelman in
\cite{taelmanwoodshole}: He describes it by his version of the function-sheaf
correspondence. Let $X(\mathbf{F}_{q})$ denote the set of $\mathbf{F}_{q}%
$-rational points of $X$, and $\operatorname*{Map}(X(\mathbf{F}_{q}%
),\mathbf{F}_{q})$ the vector space of set-theoretic maps (i.e. literally
assigning an element of $\mathbf{F}_{q}$ to each $\mathbf{F}_{q}$-rational point).

\begin{theorem}
[{Taelman, \cite[Theorem 3.6]{taelmanwoodshole}}]Suppose $X/\mathbf{F}_{q}$ is
a finite type scheme. Then there is a short exact sequence of abelian groups%
\[
0\longrightarrow R\longrightarrow K_{0}(\mathsf{CartCrys}(X))\overset
{\operatorname*{tr}}{\longrightarrow}\operatorname*{Map}(X(\mathbf{F}%
_{q}),\mathbf{F}_{q})\longrightarrow0\text{,}%
\]
given by the function-sheaf correspondence of \cite[Ch. 1, \S 2]%
{taelmanwoodshole}. The group $R$ is the subgroup generated by the differences
$[(\mathcal{F},c)]+[(\mathcal{F},c^{\prime})]-[(\mathcal{F},c+c^{\prime})]$,
where $c,c^{\prime}$ specify the action of the Cartier operators. The group
$R$ contains all $p$-th multiples, i.e. $pK_{0}(\mathsf{CartCrys}(X))$.
\end{theorem}

Loc. cit. this result is phrased for the $K_{0}$-group of the category of
$\tau$-crystals. However, this category is equivalent to $\mathsf{CartCrys}%
(X)$ by B\"{o}ckle and Pink (Theorem
\ref{thm_TauCrystalsAgreeWithCartierCrystals}). Let us temporarily remain in
the context of $\tau$-sheaves: Taelman also shows that for every
$\mathbf{F}_{q} $-rational point $x:\operatorname*{Spec}\mathbf{F}%
_{q}\hookrightarrow X$ the pushforward (in the theory of $\tau$-sheaves) of
$x_{\ast}\mathbf{1}$, is a simple $\tau$-crystal such that%
\[
\operatorname*{tr}(x_{\ast}\mathbf{1})(y)=\delta_{y=x}%
\]
for $y\in X(\mathbf{F}_{q})$. In other words, it is a canonical preimage under
his function-sheaf correspondence of the delta function supported exclusively
at the given point $x$. By the equivalence of the categories of $\tau
$-crystals and Cartier crystals, the computation $K_{0}(\mathsf{\acute{E}%
t}_{c}(X,\mathbf{F}_{q}))=\bigoplus\mathbf{Z}$ can also be performed in the
context of $\tau$-crystals. The $\#X(\mathbf{F}_{q})$ pairwise non-isomorphic
$\tau$-crystals $x_{\ast}\mathbf{1}$ then provide a canonical subset of the
simple objects.\medskip

We arrive at one of our main results:

\begin{theorem}
\label{thm_Summary_KThyOfCartCrystals}Suppose $X/\mathbf{F}_{q}$ is a smooth
scheme. Then:

\begin{enumerate}
\item There is a canonical equivalence%
\[
K(\mathsf{CartCrys}(X))\overset{\sim}{\longrightarrow}K(\mathsf{\acute{E}%
t}_{c}(X,\mathbf{F}_{q}))\text{.}%
\]

\item Moreover,%
\[
K_{m}(\mathsf{\acute{E}t}_{c}(X,\mathbf{F}_{q}))=\left\{
\begin{array}
[c]{ll}%
\text{\emph{prime-to-}}p\text{\emph{\ torsion}} & \text{for }m=2i+1\text{,}\\
0 & \text{for }m=2i\text{, }i>0\text{,}\\
\bigoplus\mathbf{Z} & \text{for }m=0\text{,}%
\end{array}
\right.
\]
where $i\geq0$, and the direct sum in the last row runs over all simple
objects of $\mathsf{CartCrys}(X)$, or equivalently perverse sheaves
$\mathsf{\acute{E}t}_{perv}(X,\mathbf{F}_{q})$. Among them, there is a
canonical set of $\#X(\mathbf{F}_{q})$ generators which surject on the
right-hand side in Sequence \ref{lTae}, while the remaining generators all map
to zero under the function-sheaf correspondence.
\end{enumerate}
\end{theorem}

\begin{proof}
\textit{(1)} By the Riemann--Hilbert correspondence of Emerton and\ Kisin
\cite{MR2071510}, in the concrete shape of Equation \ref{lcinn1}, we have an
equivalence%
\[
K(\mathsf{CartCrys}(X))\overset{\sim}{\longrightarrow}K(\mathsf{\acute{E}%
t}_{perv}(X,\mathbf{F}_{q}))
\]
since $K$-theory is not affected by switching to the opposite category. On the
right-hand side the perverse $t$-structure plays the essential r\^{o}le on the
level of categories. However, on the level of $K$-theory this washes out:
Thanks to Neeman's Theorem of the Heart (see the survey \cite[Theorem
50]{MR2181838} and the following discussion; or \cite{MR3427577}) the
following holds: Suppose $\mathsf{T}$ is a triangulated category which admits
a Waldhausen model, and $\mathsf{A}$ and $\mathsf{B}$ two abelian categories
which arise as hearts of two bounded $t$-structures on $\mathsf{T} $, then
there is an equivalence in $K$-theory, $K(\mathsf{A})\overset{\sim
}{\longrightarrow}K(\mathsf{B})$. Bounded complexes of \'{e}tale
$\mathbf{F}_{q}$-sheaves with constructible cohomology, $C_{c}^{b}%
(X_{\acute{e}t},\mathbf{F}_{q})$, provide a dg and\ Waldhausen model for
$D_{c}^{b}(X_{\acute{e}t},\mathbf{F}_{q})$. We deduce that%
\begin{equation}
K(\mathsf{\acute{E}t}_{c}(X,\mathbf{F}_{q}))\overset{\sim}{\longrightarrow
}K(\mathsf{\acute{E}t}_{perv}(X,\mathbf{F}_{q}))\label{lcin4}%
\end{equation}
by Neeman's theorem. The reader can find an analogous procedure explained and
carried out in \cite[\S 7]{MR3427577}, where it is applied to the perverse
$t$-structure on coherent sheaves.\newline\textit{(2)} The discussion of this
section implies the claims for $K(\mathsf{CartCrys}(X))$, namely Prop.
\ref{prop_K_of_CartCrys_InPrimeToPTorsion} and Taelman's result, and then use
part (1) of the proof.
\end{proof}

Note that the comparison of line \ref{lcin4} only works on the level of
$K$-theory. As abelian categories, $\mathsf{\acute{E}t}_{perv}(X,\mathbf{F}%
_{q}) $ and $\mathsf{\acute{E}t}_{c}(X,\mathbf{F}_{q})$ are very different.

\section{\label{sect_QuillenRevisited}Quillen's computation $-$ revisited}

In \textquotedblleft Higher Algebraic $K$-Theory I\textquotedblright%
\ \cite{MR0338129} Quillen studies not only the $K$-theory of schemes, but
also some examples of non-commutative rings.\ The prominent and better known
example are central simple algebras $A$ and his computation%
\[
K(X)\overset{\sim}{\longrightarrow}\coprod K(A^{\otimes i})\text{,}%
\]
where $X$ is the associated Severi--Brauer variety. This computation is, in a
way, a \textquotedblleft slightly non-commutative\textquotedblright%
\ generalization of the computation of the $K$-theory of projective space
$\mathbf{P}_{k}^{n}$.

Instead, we shall focus on the other, less well-known, example.\ It was a
motivation for this paper, because one runs into severe technical problems
when trying to generalize it, and unlike the above, it is far less clear how
to connect it to \textquotedblleft geometry\textquotedblright. Quillen gives
this example on the pages\footnote{or equivalently pages 114-115, 122-123,
depending on which of the three incompatible paginations of Quillen's paper
the reader wishes to use} 38-39: He fixes a prime power $q=p^{r}$, $r\geq1$,
defines $k:=\mathbf{F}_{q}^{\operatorname*{sep}}$ and considers the Frobenius
skew ring $A:=k[F]$ over $k$. Then he defines $D$ to be the quotient skew
field of $A$,%
\begin{equation}
D:=\operatorname*{Quot}A\text{.}\label{lBM6}%
\end{equation}

Quillen now computes the $K$-theory of this skew field:%
\begin{equation}
K_{m}(D)=\left\{
\begin{array}
[c]{ll}%
\mathbf{Z} & \text{if }m=0\\
\mathbf{Z}\oplus\mathbf{Z} & \text{if }m=1\\
(K_{2i-1}\mathbf{F}_{q})^{\oplus2} & \text{if }m=\text{even}\\
0 & \text{if }m=\text{odd.}%
\end{array}
\right.  \label{lca3}%
\end{equation}
We will now adapt this computation to smooth projective schemes over
$\mathbf{F}_{q}$: We first need to define $D$ in general, for example for a
general ring $R$ instead of $\mathbf{F}_{q}^{\operatorname*{sep}}$. We should
recall its definition: For an associative unital domain $R$, the set of
non-zero elements may satisfy the Ore conditions. If this is the case, it is
called a \emph{left} (res. \emph{right})\ \emph{Ore domain}. Then the
localization $D:=(R-\{0\})^{-1}R$ is a skew field \cite[Cor. 1.3.3]%
{MR1349108}\footnote{Note that Cohn writes \textquotedblleft
field\textquotedblright\ for what we would call a "skew field" and
\textquotedblleft$R^{\times}$\textquotedblright\ for $R-\{0\}$, while we would
write $R^{\times}$ to denote the group of units of the ring $R$.}. Now, the
following characterization holds:

\begin{lemma}
\label{Lemma_CharacterizeOreDomains}Suppose $k/\mathbf{F}_{q}$ is some field
extension. Then the Frobenius skew ring $k[F]$ is a left (or right) Ore domain
if and only if $k$ is perfect.
\end{lemma}

\begin{proof}
\cite[Prop. 2.1.6]{MR1349108}. Use that $k[F]$ is an example of a twisted
polynomial ring for the endomorphism $\sigma$ being the Frobenius $x\mapsto
x^{q}$, and the trivial derivation $\delta:=0$.
\end{proof}

\begin{remark}
There is a relation between the lack of a field of fractions and the lack of
Noetherianity. Every right Noetherian domain is a right\ Ore domain, \cite[Ch.
2, (1.15) Theorem]{MR1811901} or \cite[Ch. 4, (10.23)\ Corollary]{MR1653294}.
This is a part of Yoshino's result, Thm.
\ref{thm_Yoshino_CharacterizationNoetherian}.
\end{remark}

In particular, in the context of Quillen's paper, where $k$ is algebraically
closed and thus trivially perfect, this quotient skew field exists and he can
work with it. However, we run into a serious problem when trying to understand
the broader picture of his computation since it becomes unclear how the
definition of $D$ is to be generalized. For example, it is hopeless to try to
define such a $D$ for a curve $X/\mathbf{F}_{q}$ if already for its function
field $\mathbf{F}_{q}(X)$ the multiplicative set $\mathbf{F}_{q}(X)-\{0\}$ is
not a left or right denominator set. However, Quillen goes on to characterize
the finitely generated right $D$-modules as%
\begin{equation}
\mathsf{Mod}_{fg}(D)=\mathsf{Mod}_{fg}(A)/\mathsf{B}\text{,}\label{lBM4}%
\end{equation}
where $\mathsf{B}$ is the Serre subcategory of those $A$-modules which are
finite-dimensional as $k$-vector spaces. So, instead of worrying about
defining $D$, we can generalize its category of modules. The analogue of
$\mathsf{Mod}_{fg}(A)$ will be $\mathsf{Coh}_{\mathcal{O}_{X}[F]}(X)$ and
$\mathsf{B}$ have to be those modules which are finitely generated over
$\mathcal{O}_{X}$, but this is precisely the subcategory of Cartier modules in
the sense of Blickle and\ B\"{o}ckle \cite{MR2863904}:

\begin{definition}
\label{Def_QuillenCatQD}If $X/\mathbf{F}_{q}$ is an $F$-finite Noetherian
separated scheme, define an abelian category\footnote{I chose $\mathsf{QD}(X)$
because $\mathsf{D}(X)$ looks too much as if it were to suggest a derived
category.}%
\begin{equation}
\mathsf{QD}(X):=\mathsf{Coh}_{\mathcal{O}_{X}[F]}(X)/\mathsf{CohCart}%
(X)\text{.}\label{lbib1}%
\end{equation}

\end{definition}

This definition is justified by the following:

\begin{lemma}
If $X/\mathbf{F}_{q}$ is an $F$-finite Noetherian separated scheme, then
$\mathsf{CohCart}(X)$ is a Serre subcategory of $\mathsf{Coh}_{\mathcal{O}%
_{X}[F]}(X)$.
\end{lemma}

\begin{proof}
First of all, it is a full subcategory: Cartier modules carry, by definition,
a right $\mathcal{O}_{X}[F]$-module structure. So it only remains to show that
being coherent as an $\mathcal{O}_{X}$-module sheaf implies being finitely
presented as a right $\mathcal{O}_{X}[F]$-module sheaf. This can be checked
affine locally, and then reduces to Lemma
\ref{lem_FinGenOverRMeansFinPresOverRF}. Secondly, as $X$ is Noetherian, the
condition to be finitely generated over $\mathcal{O}_{X}$ renders this full
subcategory a Serre subcategory.
\end{proof}

\begin{example}
\label{example_Qui1}Let us return to Quillen's original example. Clearly
$X:=\operatorname*{Spec}\mathbf{F}_{q}^{\operatorname*{sep}}$ is an $F$-finite
Noetherian scheme, its sheaves identify with certain types of $\mathbf{F}%
_{q}^{\operatorname*{sep}}$-modules, notably $\mathsf{Coh}_{\mathcal{O}%
_{X}[F]}(X)\cong\mathsf{Mod}_{fg}(A)$ and $\mathsf{CohCart}(X)\cong\mathsf{B}$
in Quillen's notation. In particular,%
\[
\mathsf{QD}(X)\cong\mathsf{Mod}_{fg}(D)\text{.}%
\]

\end{example}

Quillen then proceeds to compute the $K$-theory of the skew field $D$ using
his localization sequence, based on modules over $D$ being a quotient abelian
category, as in Equation \ref{lBM4}. We can do the same for $\mathsf{QD}(X)$,
and for the same reason.

\begin{theorem}
\label{MainThmQD}Suppose $X/\mathbf{F}_{q}$ is a regular and $F$-finite
Noetherian separated scheme. Then there is a long exact sequence%
\[
\cdots\longrightarrow K_{m}(\mathsf{CohCart}(X))\overset{(\ast)}%
{\longrightarrow}K_{m}(X)\overset{\overline{\Psi}}{\longrightarrow}%
K_{m}(\mathsf{QD}(X))\longrightarrow\cdots\text{,}%
\]
where the map $(\ast)$ is $(1-F_{\ast})\circ U$.
\end{theorem}

\begin{proof}
By Quillen's localization theorem, we have the homotopy fiber sequence%
\[
K(\mathsf{CohCart}(X))\overset{\iota}{\longrightarrow}K(\mathsf{Coh}%
_{\mathcal{O}_{X}[F]}(X))\overset{\overline{(-)}}{\longrightarrow
}K(\mathsf{QD}(X))\longrightarrow+1\text{,}%
\]
straight from Definition \ref{Def_QuillenCatQD}, where $\iota$ is the
inclusion of the Serre subcategory and $\overline{(-)}$ denotes the exact
functor to the quotient abelian category. We modify this sequence in two ways:
(1)\ By Theorem \ref{thm_TwistedA1InvarianceOfKTheory} we have the Frobenius
analogue of the $\mathbf{A}^{1}$-invariance, $\Psi:K(\mathsf{Coh}%
(X))\overset{\sim}{\rightarrow}K(\mathsf{Coh}_{\mathcal{O}_{X}[F]}(X))$. We
arrive at%
\begin{equation}
K(\mathsf{CohCart}(X))\overset{\Psi^{-1}\circ\iota}{\longrightarrow
}K(\mathsf{Coh}(X))\overset{\overline{(-)}\circ\Psi}{\longrightarrow
}K(\mathsf{QD}(X))\longrightarrow+1\text{.}\label{lBM5}%
\end{equation}
By Prop. \ref{prop_CompareEquivalenceAndCategoryInclusion} the first arrow
agrees with $(1-F_{\ast})\circ U$.
\end{proof}

\begin{remark}
The fundamental r\^{o}le of the homotopy fiber of $1-F_{\ast}$ has its
counterpart in \cite{MR929766}, where $1-F^{\ast}$ has a similar function.
\end{remark}

\subsection{Implications of Parshin's conjecture}

\begin{conjecture}
[Parshin]Suppose $X/\mathbf{F}_{q}$ is a smooth projective scheme. Then
$K_{m}(X)\otimes\mathbf{Q}=0$ for $m\geq1$.
\end{conjecture}

To the best of my knowledge, the first mention of this in print is \cite[Conj.
2.4.2.3]{MR760999}. At present, this conjecture is only known for curves, or
when the $K$-theory is basically entirely known anyway, e.g. for cellular
varieties. Next, let us recall Quillen's computation, Equation \ref{lca3}, but
with rational coefficients. This means, we are talking about%
\[
K_{m}(D)_{\mathbf{Q}}=\left\{
\begin{array}
[c]{ll}%
\mathbf{Q} & \text{if }m=0\\
\mathbf{Q}\oplus\mathbf{Q} & \text{if }m=1\text{,}\\
0 & \text{if }m\geq2\text{.}%
\end{array}
\right.
\]
If we assume the validity of Parshin's Conjecture, we can see the
characteristic features of this computation repeat, in a more complicated
fashion, in the general case of varieties:

\begin{theorem}
Suppose $X/\mathbf{F}_{q}$ is a smooth, projective, and geometrically integral
scheme. Then%
\[
K_{m}(\mathsf{QD}(X))_{\mathbf{Q}}=\left\{
\begin{array}
[c]{ll}%
\mathbf{Q} & \text{if }m=0\\
\mathbf{Q}\oplus K_{0}(\mathsf{\acute{E}t}_{c}(X,\mathbf{F}_{q}))_{\mathbf{Q}}
& \text{if }m=1\text{, }^{(\ast)}\\
0 & \text{if }m\geq2\text{, }^{(\ast)}%
\end{array}
\right.
\]
where the entries marked with an asterisk $(\ast)$ are only known if $X$
satisfies Parshin's Conjecture. Unconditionally: If $x:\operatorname*{Spec}%
\kappa(x)\hookrightarrow X$ is any closed point, $[(x_{\ast}\mathcal{O}%
_{\kappa(x)})\otimes_{\mathcal{O}_{X}}\mathcal{O}_{X}[F]]$ is a basis of
$K_{0}(\mathsf{QD}(X))_{\mathbf{Q}}$.
\end{theorem}

\begin{proof}
We use Theorem \ref{MainThmQD}. \textit{(1)} For $m=0$, we obtain the
presentation%
\[
K_{0}(\mathsf{CohCart}(X))_{\mathbf{Q}}\rightarrow K_{0}(\mathsf{Coh}%
(X))_{\mathbf{Q}}\rightarrow K_{0}(\mathsf{QD}(X))_{\mathbf{Q}}\rightarrow
0\text{,}%
\]
and we know that the first arrow takes the value $(1-F_{\ast})U[\mathcal{F}]$
for Cartier modules $\mathcal{F}$. Since $U$ maps $\mathcal{F}$ to the
underlying coherent sheaf, and all coherent sheaves can be equipped with the
trivial Cartier module structure of $F$ acting by zero, the image is just
$(1-F_{\ast})$ applied to arbitrary coherent shaves. By this and the
regularity of $X$, we can also work with the right exact sequence%
\[
K_{0}(X)_{\mathbf{Q}}\overset{1-F_{\ast}}{\longrightarrow}K_{0}(X)_{\mathbf{Q}%
}\rightarrow K_{0}(\mathsf{QD}(X))_{\mathbf{Q}}\rightarrow0\text{.}%
\]
Using Grothendieck-Riemann-Roch, $K_{0}(X)\otimes\mathbf{Q}\cong\coprod
_{i\geq0}\operatorname*{CH}_{i}(X)_{\mathbf{Q}}$, and the Frobenius
pushforward acts as $q^{i}$ on $\operatorname*{CH}_{i}(X)_{\mathbf{Q}}$ (by
direct computation, or by \cite[\S 1.5.3, Prop. 2]{MR751733}). For $i\geq1$,
the element $1-q^{i}$ is invertible in the rationals, so $1-F_{\ast}$ acts as
an isomorphism on these summands, while it is the zero map for
$\operatorname*{CH}_{0}(X)_{\mathbf{Q}}$. Thus, we get an isomorphism%
\begin{equation}
\operatorname*{CH}\nolimits_{0}(X)_{\mathbf{Q}}\overset{\cong}{\longrightarrow
}K_{0}(\mathsf{QD}(X))\otimes\mathbf{Q}\text{.}\label{liaps1}%
\end{equation}
The zero cycles sit in an exact sequence, $0\rightarrow A_{0}(X)\rightarrow
\operatorname*{CH}\nolimits_{0}(X)\rightarrow\mathbf{Z}\rightarrow0$, with the
degree map. By\ Kato--Saito unramified class field theory, specifically
\cite[Theorem 1]{MR717824}, the reciprocity map sends this sequence to
$0\rightarrow\pi_{1}^{\operatorname*{geom}}(X,\ast)_{ab}\rightarrow\pi
_{1}^{\acute{e}t}(X,\ast)_{ab}\rightarrow\operatorname*{Gal}(\mathbf{F}%
_{q}^{\operatorname*{sep}}/\mathbf{F}_{q})\rightarrow0$ such that it induces
an isomorphism on the left term, and the profinite completion $\mathbf{Z}%
\hookrightarrow\widehat{\mathbf{Z}}$ on the right. Moreover, by
Katz--Lang\ Finiteness \cite[Theorem 2]{MR659153}, applied to the structural
morphism $X\rightarrow\operatorname*{Spec}\mathbf{F}_{q}$, the geometric part
$\pi_{1}^{\operatorname*{geom}}(X,\ast)_{ab}$ is finite (the field
$\mathbf{F}_{q}$ is clearly accessible in the sense loc. cit. since it is even
finite over its prime field). Hence, line \ref{liaps1} implies
$\operatorname*{CH}\nolimits_{0}(X)_{\mathbf{Q}}\cong\mathbf{Q}$ and that
$K_{0}(\mathsf{QD}(X))\otimes\mathbf{Q}$ is one-dimensional. Finally, any
closed point generates $\operatorname*{CH}\nolimits_{0}(X)_{\mathbf{Q}}$ and
unwinding the maps, this gives the explicit generator $[(x_{\ast}%
\mathcal{O}_{\mathbf{F}_{q}})\otimes_{\mathcal{O}_{X}}\mathcal{O}_{X}%
[F]]$.\newline\textit{(2)} For $m=1$, we get%
\begin{equation}
\cdots\rightarrow K_{1}(X)_{\mathbf{Q}}\rightarrow K_{1}(\mathsf{QD}%
(X))_{\mathbf{Q}}\rightarrow K_{0}(\mathsf{CohCart}(X))_{\mathbf{Q}}%
\overset{(\ast)}{\rightarrow}K_{0}(X)_{\mathbf{Q}}\text{.}\label{lvv1}%
\end{equation}
By\ Parshin's Conjecture, $K_{1}(X)_{\mathbf{Q}}=0$, so it follows that
$K_{1}(\mathsf{QD}(X))_{\mathbf{Q}}$ is the kernel of the morphism $(\ast)$.
We compute this kernel as follows: By Prop.
\ref{Prop_SplitingOfKTheoryOfCohCart} we have the exact sequence%
\[
K_{1}(\mathsf{CartCrys}(X))\rightarrow K_{0}(\mathsf{Coh}(X))\rightarrow
K_{0}(\mathsf{CohCart}(X))\rightarrow K_{0}(\mathsf{CartCrys}(X))\rightarrow0
\]
and by Prop. \ref{prop_K_of_CartCrys_InPrimeToPTorsion} the group
$K_{1}(\mathsf{CartCrys}(X))$ is pure torsion. We obtain the commutative
diagram%
\[%
{
\xymatrix@C=0.1in
{
0 \ar[r] & K_{0}(\mathsf{Coh}(X))_{\mathbf{Q}} \ar[r] \ar[d]_{1-F_{\ast}}
& K_{0}(\mathsf{CohCart}(X))_{\mathbf{Q}} \ar[r] \ar[d]_{(1-F_{\ast})U}
& K_{0}(\mathsf{CartCrys}(X))_{\mathbf{Q}} \ar[r] \ar[d] & 0 \\
0 \ar[r]  & K_{0}(\mathsf{Coh}(X))_{\mathbf{Q}} \ar[r]^{\sim} & K_{0}%
(\mathsf{Coh}(X))_{\mathbf{Q}} \ar[r] &  0  \ar[r] & 0
}
}%
%EndExpansion
\]
and since $1-F_{\ast}$ acts as $1-q^{i}$ on the rationalized
$\operatorname*{CH}_{i}$ summand, the left and middle downward arrow have
$\operatorname*{CH}\nolimits_{0}(X)_{\mathbf{Q}}$ both as kernel and cokernel.
The snake lemma yields a long exact sequence%
\[
0\rightarrow\operatorname*{CH}\nolimits_{0}(X)_{\mathbf{Q}}\rightarrow
K_{1}(\mathsf{QD}(X))_{\mathbf{Q}}\rightarrow K_{0}(\mathsf{CartCrys}%
(X))_{\mathbf{Q}}\rightarrow\operatorname*{CH}\nolimits_{0}(X)_{\mathbf{Q}%
}\overset{\sim}{\rightarrow}\operatorname*{CH}\nolimits_{0}(X)_{\mathbf{Q}%
}\rightarrow0
\]
and the last arrow is an isomorphism since both downward arrows have the same
image, so the snake map must be the zero map. Finally, use the
Riemann--Hilbert correspondence to identify the $K$-theory of the Cartier
crystals with the one of $\mathsf{\acute{E}t}_{c}$ (Theorem
\ref{thm_Summary_KThyOfCartCrystals}).\newline\textit{(3)}\ For $m\geq2$,
Theorem \ref{MainThmQD} tells us that%
\begin{equation}
\cdots\rightarrow K_{m}(X)_{\mathbf{Q}}\rightarrow K_{m}(\mathsf{QD}%
(X))_{\mathbf{Q}}\rightarrow K_{m-1}(\mathsf{CohCart}(X))_{\mathbf{Q}%
}\rightarrow\cdots\label{lca10}%
\end{equation}
is exact. By Parshin's Conjecture, the rationalized $K$-groups $K_{m}%
(X)_{\mathbf{Q}}$ vanish. Moreover, in%
\[
K_{m-1}(\mathsf{Coh}(X))_{\mathbf{Q}}\rightarrow K_{m-1}(\mathsf{CohCart}%
(X))_{\mathbf{Q}}\rightarrow K_{m-1}(\mathsf{CartCrys}(X))_{\mathbf{Q}}%
\]
of Prop. \ref{Prop_SplitingOfKTheoryOfCohCart} the left term vanishes by
Parshin's Conjecture, and the right one by Prop.
\ref{prop_K_of_CartCrys_InPrimeToPTorsion}. Thus, $K_{m-1}(\mathsf{CohCart}%
(X))_{\mathbf{Q}}=0$ and thus $K_{m}(\mathsf{QD}(X))_{\mathbf{Q}}=0$ by the
exactness of line \ref{lca10}. This finishes the proof.
\end{proof}

\begin{acknowledgement}
I heartily thank M. Blickle, L.\ Kindler and L. Taelman for very interesting
conversations. I would also like to thank G. B\"{o}ckle, M. Emerton,
D.\ Grayson, L.\ Hesselholt and M. Wendt for their comments and answering my
questions. I would also like to thank the participants of the 2016 Oberwolfach
workshop on Algebraic $K$-theory for the inspiring atmosphere.
\end{acknowledgement}

\bibliographystyle{amsalpha}
\bibliography{ollinewbib}

Date:
{\today}

\end{document}